\documentclass[12pt]{amsart}

\usepackage{epsfig,color}
\usepackage{blindtext}
\usepackage{hyperref}
\usepackage{graphicx}
\usepackage{enumitem} 
\usepackage{url}
\usepackage{amssymb}
\usepackage{graphicx,import}
\usepackage{comment}
\usepackage{ulem}
\usepackage{outlines}
\usepackage{esint}
\usepackage{verbatim}
\usepackage{mathrsfs}
\usepackage{mathtools}
\usepackage{amsmath}
\usepackage{dsfont}
\usepackage{comment}

\theoremstyle{definition}

\newtheorem{theorem}{Theorem}[section]
\newtheorem{definition}[theorem]{Definition}

\newtheorem{lemma}[theorem]{Lemma}
\newtheorem{remark}[theorem]{Remark}
\newtheorem{corollary}[theorem]{Corollary}
\newtheorem{example}[theorem]{Example}

\newtheorem*{remark*}{Remark}

\numberwithin{equation}{subsection}
\numberwithin{theorem}{subsection}

\newcommand{\mf}{\mathbf}
\newcommand{\mc}{\mathcal}
\newcommand{\mb}{\mathbb}
\newcommand{\mr}{\mathrm}

\newcommand{\dist}{\mathrm{dist}}

\newcommand{\pa}{\partial}

\DeclareMathOperator{\area}{Area}

\DeclareMathOperator{\di}{div}

\title[Nonlocal Energy Functionals \& Determinantal Point Processes]{Nonlocal Energy Functionals and Determinantal Point Processes on Non-Smooth Domains}

\date{\today}

\author{Zhengjiang Lin}

\address{Courant Institute of Mathematical Sciences, New York University, 251 Mercer Street, New York, NY10012, USA}

\email{malin@nyu.edu}

\begin{document}

\begin{abstract}
	Given a nonnegative integrable function $J$ on $\mb{R}^n$, we relate the asymptotic properties of the nonlocal energy functional
		\begin{equation*}
			\int_{\Omega} \int_{\Omega^c}   J \bigg(\frac{x-y}{t}\bigg) \ dx dy
		\end{equation*}
	as $t \to 0^+$ with the boundary properties of a given domain $\Omega \subset \mathbb{R}^n$. 
	Then, we use these asymptotic properties to study the fluctuations of many determinantal point processes, and show that their variances measure the Minkowski dimension of $\pa \Omega$.
\end{abstract}

\maketitle

\section{Introduction}\label{Sec: Introduction}

	A point process on $\mb{R}^n$ is defined as a measurable map from a probability space $(\Omega, \mc{B}(\Omega) ,\mc{P})$ to random configurations of points $\{x_i{\}}_{i \in I}$, where $I$ is a countable index set.
	More precisely, a point process on $\mb{R}^n$ is a probability measure on the set of locally finite counting measures on $\mb{R}^n$.
	Usually, a point process $X$ takes the form
		\begin{equation}
			X = \sum_{i \in I} \delta_{x_i},
		\end{equation}
	where $\delta_x$ is the dirac measure at point $x \in \mb{R}^n$, and the configurations of $x_i$'s follow a probability distribution.
	Locally finiteness means that, almost surely, $X(K) < \infty$ for any compact set $K \subset \mb{R}^n$.
	We will define this precisely in Section~\ref{Sec: Det Process} and readers can also see~~\cite{AGZ10} for more details. 
	We define $X(\varphi) \equiv \sum_{i \in I} \varphi(x_i)$ for any appropriate function $\varphi$, provided that the summation on the right-hand side is well-defined.
	In particular, if $A \subset \mb{R}^n$ is a measurable set, we define $X(A) \equiv X(\mathds{1}_A)$ to denote the number of $x_i$'s that are contained in the set $A$. Here, $\mathds{1}_A$ is the indicator function of $A$.

	We are interested in the variance $\mr{Var}(X(\varphi))$. In most of the models from either physics or probability, there is a nonnegative integrable function $J = J_{X}$ on $\mb{R}^n$, depending on the point process $X$, such that
		\begin{equation}\label{Eqn: Variance is a Nonlocal Functional}
			\mr{Var}(X(\varphi)) = \int_{\mb{R}^n \times \mb{R}^n} { \big|\varphi(x) - \varphi(y) \big| }^2 \cdot J(x-y) \ dx dy .
		\end{equation}
	Such functionals have already been extensively studied in analysis and geometry to investigate the asymptotic behaviors of a particular family of $\{J_t{\}}$'s as $t$ converges to a critical number in $\mb{R}$ or $\infty$. 
	For example, in the fractional Sobolev spaces $W^{t,p}(\mb{R}^n)$ (or $W^{t,p}(A)$ with $A$ being a smooth bounded domain in $\mb{R}^n$), a measurable function $\varphi$ is said to be in $W^{t,p}(A)$ if the functional
		\begin{equation}
			{||\varphi||}^p _{W^{t,p}} = \int_A \int_A \frac{{|\varphi(x) - \varphi(y)|}^p}{{|x-y|}^{n+tp}} \ dx dy
		\end{equation}
	is finite. One would like to see the asymptotic behaviors (orders) of such functionals when $t \to 1^-$, which diverges for smooth nonconstant $\varphi$. For further discussions on $W^{t,p}$-spaces, one may refer to~~\cite{BBM01}, where the authors associated the limit of the functional (\ref{Eqn: Variance is a Nonlocal Functional}) for a general family of $\{J_t{\}}$ to the $W^{1,2}$-norm of $\varphi$. We will also briefly discuss the Sobolev spaces in Section~\ref{Sec: Limit Sobolev}.
	
	When $\varphi = \mathds{1}_{\Omega}$ is the indicator function of a measurable set $\Omega \subset \mb{R}^n$, the functional (\ref{Eqn: Variance is a Nonlocal Functional}) is closely related to the study of the nonlocal perimeter and nonlocal minimal surfaces~~\cite{CRS10, MRT19}.
	In this manuscript, we also focus on situations where $\varphi = \mathds{1}_{\Omega}$.
	Our primary purpose is to study the asymptotic behavior (order) of the functional
		\begin{equation}
			\mc{J}_t (\Omega) \equiv \frac{1}{2}\int_{\mb{R}^n \times \mb{R}^n} { \big|\mathds{1}_{\Omega}(x) - \mathds{1}_{\Omega}(y)  \big| }^2 \cdot J \bigg(\frac{x-y}{t}\bigg) \ dx dy. = \int_{\Omega} \int_{\Omega^c}   J \bigg(\frac{x-y}{t}\bigg) \ dx dy,
		\end{equation}
	as $t \to 0^{+}$. 
	Intuitively, if $J$ has a fast-decaying tail, then as $t \to 0^{+}$, only those pairs $(x,y)$ with $\dist(x,y) \sim t$, or equivalently, near the boundary $\pa \Omega$ within a distance $\sim t$, will contribute to the first order term in $\mc{J}_t (\Omega)$.
    We note that although $\mathds{1}_{\Omega}$ is not weakly differentiable, it can be in the space of functions of bounded total variation (BV space). 
    This functional space includes indicator functions of bounded domains with Lipschitz boundaries.
	The asymptotic behavior for functionals with BV test functions was also studied in~~\cite{D02}, which we will also discuss in Section~\ref{Sec: Limit BV}.

	In addition to the domains with ``locally flat'' boundaries mentioned above, one of our novel parts is that we will also discuss cases when $\pa \Omega$ does not locally  ``look like'' a hyperplane, including the well-known Koch snowflake (Figure~\ref{Fig: Snowflake}) and a family of fractals constructed in a similar way as the standard Koch snowflake.
	Also, for general $\Omega \subset \mb{R}^n$ whose boundary $\pa \Omega$ has Minkowski dimension (or Hausdorff dimension) larger than $n-1$, we can give a way to quantify the rate of decay with respect to $t$, i.e., we will show that  $\mc{J}_t (\Omega) \sim t^{\beta}$, where $\beta$ depends on the Minkowski dimension of $\pa \Omega$ and $n$.
	Moreover, if the boundary $\pa \Omega$ is self-similar in some way, it is possible to show the existence of the limit $\lim_{t \to 0^+} \mc{J}_t(\Omega) / t^{\beta}$. See our Theorem~\ref{Thm: Main Upper Lower Bounds} and Theorem~\ref{Thm: Main Twistedflake limit}.

	Our technical sections, Section~\ref{Sec: Bounds} and Section~\ref{Sec: All Limits}, do not rely on any specific probability models.
	In Section~\ref{Sec: Main Results}, we will present our main results without referring to probability. And we will present the main results in probability language in Section~\ref{Sec: Main Results in probability}.
	At the end of Section~\ref{Sec: Main Results in probability}, we will briefly discuss those $J$ which can possibly change signs and give an example in Remark~\ref{Rmk: J Change Signs}.
	At the end of this manuscript, Section~\ref{Sec: Det Process}, we will give the proofs of the results in Section~\ref{Sec: Main Results in probability} for the complex Ginibre ensemble.

\subsection{Results on Nonlocal Energy Functionals}\label{Sec: Main Results}

	\indent

	Recall that for a bounded measurable set $A \subset \mb{R}^n$, we define the $\alpha$-dimensional upper Minkowski content of $A$ as
		\begin{equation}
			\overline{\mc{M}}^{\alpha}(A) \equiv \limsup_{t \to 0^{+}} \frac{|B(A,t)|}{t^{n-\alpha}},
		\end{equation}
	and the $\alpha$-dimensional lower Minkowski content of $\Omega$ as
		\begin{equation}
			\underline{\mc{M}}^{\alpha}(A) \equiv \liminf_{t \to 0^{+}} \frac{|B(A,t)|}{t^{n-\alpha}}.
		\end{equation}
	Here,
		\begin{equation}
			B(A,t) \equiv \{ x \in \mb{R}^n \ | \ \dist(x,A) = \inf_{y \in A} \dist(x,y) <t {\}},
		\end{equation}
	and $|B(A,t)|$ is the Lebesgue measure (volume) of $B(A,t)$.
	The upper and lower Minkowski dimensions of $A$ are defined by 
		\begin{equation}
			\overline{\dim}_{\mc{M}}(A) = \sup \{ \alpha \geq 0 \ | \  \overline{\mc{M}}^{\alpha}(A) = \infty {\}} = \inf \{ \alpha \geq 0 \ | \  \overline{\mc{M}}^{\alpha}(A) = 0 {\}},
		\end{equation}
	and
		\begin{equation}
			\underline{\dim}_{\mc{M}}(A) = \sup \{ \alpha \geq 0 \ | \  \underline{\mc{M}}^{\alpha}(A) = \infty {\}} = \inf \{ \alpha \geq 0 \ | \  \underline{\mc{M}}^{\alpha}(A) = 0 {\}}.
		\end{equation}
	Now, we assume that $\Omega \subset \mb{R}^n$ is a bounded measurable set with $|\Omega| >0$, and with a topological boundary $\pa \Omega$ such that 
		\begin{equation}\label{Eqn: Minkowski Contents Bounds}
			0 < M_2 \leq \underline{\mc{M}}^{\alpha}(\pa \Omega) \leq \overline{\mc{M}}^{\alpha}(\pa \Omega) \leq M_1 < \infty, 
		\end{equation}
	where $M_1, M_2$ are two positive constants and $\alpha \in [n-1,n]$. This actually implies that $\pa \Omega$ is of Minkowski dimension $\alpha$.

	For a nonnegative function $J \in L^1(\mb{R}^n)$, let use consider the functional
		\begin{equation}
			\mc{J}_t (\Omega) \equiv \frac{1}{2}\int_{\mb{R}^n \times \mb{R}^n} { \big|\mathds{1}_{\Omega}(x) - \mathds{1}_{\Omega}(y)  \big| }^2 \cdot J \bigg(\frac{x-y}{t}\bigg) \ dx dy. = \int_{\Omega} \int_{\Omega^c}   J \bigg(\frac{x-y}{t}\bigg) \ dx dy,
		\end{equation}
	where $\Omega^c = \mb{R}^n \backslash \Omega$.
	For all theorems in Section~\ref{Sec: Main Results}, we will assume that $\int_{\mb{R}^n} J(z) \cdot {|z|}^n \ dz <\infty$. 
	Notice that this together with the assumption that $J \in L^1 (\mb{R}^n)$ shows that there is a constant $C_J>0$, such that for any $\gamma \in [0,n]$, $\int_{\mb{R}^n} J(z) \cdot {|z|}^{\gamma} \ dz \leq C_J < \infty $.

	\begin{theorem}\label{Thm: Main Upper Lower Bounds}
		\textit{
		For the kernel $J$, we also assume that there are two positive constants $c_J, a_J$ such that when $|x| \leq a_J$, $J(z) \geq c_J >0$.
		For the domain $\Omega$, we further assume that there is a positive constant $D_{\pa \Omega}$, such that  for all $x \in \pa \Omega$ and all $t \in (0,1)$, $\min\{|B(x,t) \cap \Omega|, |B(x,t) \cap \Omega^c| {\}} \geq D_{\pa \Omega} \cdot t^n$.
		Then, there is a constant $O_1$ depending on $n, C_J$, and a constant $O_2$ depending on $n, a_J,c_J,D_{\pa \Omega}$, such that
			\begin{equation}
					0 < O_2 \cdot M_2 \leq  \liminf_{t \to 0^+} \frac{\mc{J}_t (\Omega)}{t^{2n-\alpha}} \leq \limsup_{t \to 0^+} \frac{\mc{J}_t (\Omega)}{t^{2n-\alpha}} \leq O_1 \cdot M_1 < \infty. 
			\end{equation}
		}
	\end{theorem}
	We remark that the density lower bounds condition for points on $\pa \Omega$, i.e., the existence of $D_{\pa \Omega}$, is satisfied by a class of domains known as NTA (nontangentially accessible) domains, which encompass domains with Lipschitz boundaries, quasiballs, and many self-similar fractals.
	This concept was first introduced by Jerison and Kenig in~~\cite{JK82}, and one can also find more related domains in the recent survey paper~~\cite{T17}.

	For sets of finite perimeter (also known as Caccioppoli sets; see Definition~\ref{Def: Finite Perimeter}), the limit was explicitly computed in~~\cite{D02}. Let us cite a theorem in~~\cite{D02}. In particular, sets of finite perimeter include bounded domains with Lipschitz boundaries.
	\begin{theorem}[\cite{D02}, Theorem 1]\label{Thm: Main Finite Perimeter Limit}
		\textit{
		If $\Omega$ is a set of finite perimeter and $J(z) = J(|z|)$ is radially symmetric, then there is a positive dimensional constant $K(n)$, such that
			\begin{equation}
				\lim_{t \to 0^{+}} \frac{\mc{J}_t (\Omega)}{t^{n+1}}  = 2 K(n) \cdot ||(J(z) \cdot |z|) || _{L^1(\mb{R}^n)} \cdot \mc{H}^{n-1}(\pa^* \Omega) .
			\end{equation}
		}
	\end{theorem}
	Here, $K(n)$ is explicitly computed in (\ref{Eqn: K(n) for Sets of Finite Perimeter}), and $\mc{H}^{n-1}$ denotes the $(n-1)$-dimensional Hausdorff measure (or surface measure for smooth hypersurfaces). $\pa ^* \Omega$ is called the reduced boundary of $\Omega$.
	By Federer's theorem (see, for example, Theorem 16.2 in~~\cite{M12}), up to a set of zero $(n-1)$-dimensional Hausdorff measure, this reduced boundary $\pa^* \Omega$ equals the essential boundary $\pa^e \Omega$ of $\Omega$, which is 
		\begin{equation}
			\pa^e \Omega \equiv \mb{R}^n \backslash (\Omega^{(0)} \cup \Omega^{(1)}),
		\end{equation}
	where we define the set of points of density $\beta$ of $\Omega$ for $\beta \in [0,1]$ as
		\begin{equation}
			\Omega^{(\beta)} \equiv \bigg\{ x \in \mb{R}^n \ \bigg| \ \lim_{r \to 0^+}\frac{|\Omega \cap B(x,r)|}{|B(x,r)|} = \beta \bigg{\}}.
		\end{equation}
	See more discussions on BV functions and sets of finite perimeter in Section~\ref{Sec: Limit BV}.

	The last theorem is about a family of self-similar fractals constructed in as similar way as the standard Koch snowflake, which we call snowflakes of scales $\eta$. See more detailed discussions on Koch snowflakes in Section~\ref{Sec: Limit Fractal}, Figure~\ref{Fig: Snowflake}, Figure~\ref{Fig: Twistedflake}, and Figure~\ref{Fig: Constructflake}.
	\begin{theorem}\label{Thm: Main Twistedflake limit}
		\textit{
		Let $\Omega = \Omega(\eta) \subset \mb{R}^2 $ be the Koch snowflake of a scale $\eta >1$ with $\eta$ satisfying the condition (\ref{Eqn: Non-Lattice Condition}), then the limit
			\begin{equation}
				\lim_{t \to 0^+} {\frac{\mc{J}_t (\Omega)}{t^{4- \alpha(\eta)}}} 
			\end{equation}
		exists, where $\alpha(\eta) \in (1,2)$ is the Minkowski dimension of $\pa \Omega$ and satisfies (\ref{Eqn: Dimension of Twistedflake}).
		}
	\end{theorem}
	Combining Theorem~\ref{Thm: Main Twistedflake limit} with Theorem~\ref{Thm: Main Upper Lower Bounds}, the limit in Theorem~\ref{Thm: Main Twistedflake limit} is also comparable to the Minkowski contents of $\pa \Omega$.
	We also remark that this Theorem~\ref{Thm: Main Twistedflake limit} is an illustration on how to establish the existence of such a limit, which can be extended to other domains with self-similar fractal boundaries.
	Moreover, the algebraic condition (\ref{Eqn: Non-Lattice Condition}) is satisfied by almost all $\eta>1$ and hence $\alpha(\eta)$ can potentially take on almost any value in $(1,2)$ in Theorem~\ref{Thm: Main Twistedflake limit}.
    And for other constructions of self-similar fractal boundaries, one can obtain similar algebraic conditions like (\ref{Eqn: Non-Lattice Condition}). The union of all those possible $\alpha(\eta)$ will be the whole $(1,2)$.
	See more discussions in Section~\ref{Sec: Limit Fractal}.
	
	Finally, we note that in the functional $\mc{J}_t (\Omega)$, we considered a double integral over a domain $\Omega$ and its complement $\Omega^c$. In Section~\ref{Sec: Bounds} and some parts of Section~\ref{Sec: All Limits}, we will present the lemmas and proofs for double integrals over two disjoint domains $\Omega_1$ and $\Omega_2$, with similar assumptions near $\pa \Omega_1 \cap \pa \Omega_2$. 
	This provides a more general statement of our results.

\subsection{Results on Fluctuations of Determinantal Point Processes}\label{Sec: Main Results in probability}
	\indent

	This section contains the applications of results in Section~\ref{Sec: Main Results} on determinantal point processes, which was first introduced by Macchi~~\cite{M75} and originally called the fermion point process.
	Let $X$ be a determinantal point process on $\mb{R}^n$, and let $A_1 , A_2, \ldots, A_p$ be disjoint measurable sets in $\mb{R}^n$, then
		\begin{equation}\label{Eqn: Expectation Det Pt Process}
			\mc{E} \big[ X(A_1)X(A_2) \cdots X(A_p) \big] = \int_{A_1 \times A_2 \cdots \times A_p } \det {\big( K(x_i,x_j) \big)}_{1 \leq i,j \leq p} \ d\mu(x_1) \cdots d\mu(x_p).
		\end{equation}
	Here, $\mc{E}(\cdot)$ is the expectation, $\mu$ is a reference probability measure on $\mb{R}^n$, $K$ is a measurable function, and for many physical models, the matrix ${\big( K(x_i,x_j) \big)}_{1 \leq i,j \leq p}$ is Hermitian for any $p \in \mb{Z}_+$.
	Such a determinantal point process $X$ is called a determinantal point process with kernel $K$ with respect to $\mu$.

	Typical examples for determinantal point processes on $\mb{R}^2 = \mb{C}$ can come from random matrix theory, where we usually consider $X_N = \sum_{i=1} ^N \delta_{x_i}$ with $x_i$'s being eigenvalues of $N \times N$ random matrices. For example, the Gaussian unitary ensemble (GUE) and the complex Ginibre ensemble~~\cite{AGZ10}.

	On the other hand, given a function $K$ and a reference measure $\mu$ on $\mb{R}^n$, if $K$ defines a self-adjoint integral operator from $L^2(\mb{R}^n, d \mu)$ to $L^2(\mb{R}^n, d \mu)$, and is locally trace class with all eigenvalues in $[0,1]$, then there exists a determinantal point process, uniquely in law, with such a $K$ satisfying (\ref{Eqn: Expectation Det Pt Process}), see Corollary 4.2.21 of~~\cite{AGZ10}.
	In~~\cite{RV07}, readers can see three models following this logic.
	For determinantal point processes on higher dimensional $\mb{R}^n$, readers can see~~\cite{TSZ08}, where the authors called them Fermi-shells point processes.
	Soshnikov~~\cite{S00} discussed many models of determinantal point processes from random matrix theory, statistical mechanics, and quantum mechanics.

	For most of determinantal point processes, including those aforementioned models, the variance has a more concise formula:
	by (\ref{Eqn: Expectation Det Pt Process}),  for any appropriate test function $\varphi$, it has the form
		\begin{equation}\label{Eqn: Typical Variance}
			\mr{Var}(X(\varphi)) = \frac{1}{2}\int_{\mb{R}^n \times \mb{R}^n} { \big|\varphi(x) - \varphi(y) \big| }^2 \cdot {|K(x,y)|}^2   \ d\mu(x)  d\mu(y).
		\end{equation}
	In particular, for some of the aforementioned models, such as the planar model in~~\cite{RV07}, and the Fermi-shells point processes in~~\cite{TSZ08}, it can be shown that $\frac{d\mu(x)}{dx} = \xi(x) \in L^1(\mb{R}^n , dx)$ is nonnegative, and ${|K(x,y)|}^2 \xi(x) \xi (y)$ is a function of $x-y$, i.e., has the form $J(x-y)$ for a nonnegative integrable function $J$.
	Here, $\frac{d\mu(x)}{dx} = \xi(x)$ is the Radon-Nikodym derivative of $\mu$ with respect to the Lebesgue measure $dx$ on $\mb{R}^n$.
	The functionals then become 
		\begin{equation}
			\int_{\mb{R}^n \times \mb{R}^n} { \big|\varphi(x) - \varphi(y) \big| }^2 \cdot J(x-y) \ dx dy ,
		\end{equation}
	for which we discussed the asymptotic orders of $\mc{J}_t(\Omega)$ for the case $\varphi = \mathds{1}_{\Omega}$ in Section~\ref{Sec: Main Results}.
	On the other hand, the asymptotic behaviors of the variances for a family of $(X_t, \varphi_t  = \mathds{1}_{\Omega_t})$ will imply some forms of central limit theorems.
	Roughly speaking, under some mild assumptions on those determinantal processes $X_t$, if one can further obtain that $\mr{Var}(X_t(\Omega_t)) \to \infty$ as the parameter $t$ tending to some number in $\mb{R}$ or $\infty$, then $[X_t(\Omega_t) - \mc{E}(X_t(\Omega_t))] / \sqrt{\mr{Var}(X_t(\Omega_t))}$ converges in law to the normal distribution.
	See~\cite{CL95}, Theorem 8 of~\cite{S00}, or Theorem 4.2.25 of~~\cite{AGZ10}.
	We then can obtain the following direct corollaries.

	\begin{corollary}\label{Cor: Main Probability Models 1}
		\textit{
		Let $X$ be a determinantal point process on $\mb{R}^n$ with kernel $K$ with respect to $\mu$, such that ${|K(x,y)|}^2 \xi (x) \xi (y) = J(x-y)$ is a nonnegative integrable function, where $\xi(x)$ is the derivative of $\mu$ with respect to $dx$ on $\mb{R}^n$.
		Then we have the following cases:
			\begin{itemize}
				\item [(i)] Assume that $\Omega$ is a set of finite perimeter in $\mb{R}^n$ and $J(x) = J(|x|)$ is radially symmetric. Then, there is a computable positive constant $T = T(n , ||(J(x)\cdot |x|) ||_{L^1(\mb{R}^n , dx)} )$, such that
						\begin{equation}
							\lim_{t \to 0^{+}} t^{n-1} \cdot \mr{Var}( X (\Omega / t))  = T \cdot \mc{H}^{n-1}(\pa^* \Omega) .
						\end{equation}
					    Here, $\Omega / t \equiv \{ x / t \in \mb{R}^n \ | \ x \in \Omega{\}}$.
				\item [(ii)] Assume that $\pa \Omega$ is of Minkowski dimension $\alpha \in [n-1,n]$ and satisfies (\ref{Eqn: Minkowski Contents Bounds}). We also assume that $J$ and $\Omega$ satisfy the assumptions we made for Theorem~\ref{Thm: Main Upper Lower Bounds}, then 
						\begin{equation}
							0<\liminf_{t \to 0^+} \big[t^{\alpha} \cdot \mr{Var}( X (\Omega / t)) \big] \leq \limsup_{t \to 0^+} \big[ t^{\alpha} \cdot \mr{Var}( X (\Omega / t)) \big] < + \infty. 
						\end{equation}
					     Moreover, $t^{\alpha} \cdot \mr{Var}( X (\Omega / t))$ is comparable to the Minkowski contents of $\pa \Omega$.
				\item [(iii)] When $n=2$, assume that $\Omega$ is the Koch snowflake of a scale $\eta >1$ with $\eta$ satisfying the condition (\ref{Eqn: Non-Lattice Condition}), then the limit 
						\begin{equation}
							\lim_{t \to 0^+} t^{\alpha(\eta)} \cdot \mr{Var}( X (\Omega / t))
						\end{equation}
					    exists, where $\alpha(\eta)$ is the Minkowski dimension of $\pa \Omega$. Moreover, the limit is comparable to the Minkowski contents of $\pa \Omega$.
			\end{itemize}
		In all these three cases, we then obtain that the normalization of $X(\Omega/t)$ converges in law to the normal distribution.
		}
	\end{corollary}

	Corollary~\ref{Cor: Main Probability Models 1} is a direct result according to those theorems in Section~\ref{Sec: Main Results} and the aforementioned central limit theorems results in~\cite{CL95,S00,AGZ10}.
	Also, part (iii) of Corollary~\ref{Cor: Main Probability Models 1} also holds true for some other domains $\Omega$ with self-similar fractal boundaries $\pa \Omega$ in $\mb{R}^n$. See the discussions in Section~\ref{Sec: Limit Fractal} and the proofs of our Theorem~\ref{Thm: Snowflake Limit} Theorem~\ref{Thm: Twistedflake limit}.

        In Corollary~\ref{Cor: Main Probability Models 1}, the condition that ${|K(x,y)|}^2 \xi (x) \xi (y)$ has a form of $J(x-y)$ is satisfied in the planar model in~~\cite{RV07} and the Fermi-shells point processes in~~\cite{TSZ08}. And in both cases, $J(x-y) = J(|x-y|)$ is radially symmetric. On the other hand, if $X_N$ is the determinantal point process associated with the complex Ginibre emsemble on $\mb{C} = \mb{R}^2$, the corresponding ${|K_N(x,y)|}^2 \xi_N (x) \xi_N (y)$ is not a function of $x-y$ anymore. Instead, as $N \to + \infty$, ${|K_N(x,y)|}^2 \xi_N (x) \xi_N (y)$ will tend to the planar model in~~\cite{RV07} when $|x|<1, |y|<1$. See more detailed discussions in Section~\ref{Sec: Det Process}.

        \begin{corollary}\label{Cor: Main Probability Models 2}
            \textit{
            Let $X_N$ be the determinantal point process associated the complex Ginibre ensemble and let $\Omega \subset \mb{C} = \mb{R}^2$ be a domain whose closure $\overline{\Omega}$ is contained in the unit ball $B(0,1) \equiv \{x \in \mb{R}^2 \ | \ |x| <1{\}}$. Then we have the following cases:
                \begin{itemize}
                    \item [(i)] Assume that $\Omega$ is a set of finite perimeter, then there is a computable positive constant $T$ such that
                                \begin{equation}
                    				\lim_{N \to \infty} \frac{ \mr{Var}(X_N(\Omega))}{\sqrt{N}} = T \cdot \mc{H}^{1}(\pa^* \Omega).
        			        \end{equation}
                    \item [(ii)] Assume that $\Omega$ satisfies the assumptions we made for Theorem~\ref{Thm: Main Upper Lower Bounds}  and $\pa \Omega$ is of Minkowski dimension $\alpha \in [1,2]$ and satisfies (\ref{Eqn: Minkowski Contents Bounds}), then 
						\begin{equation}
            				 0 < \liminf_{N \to + \infty} \frac{ \mr{Var}(X_N(\Omega))}{{N}^{(\alpha /2)}} =
                        \limsup_{N \to + \infty} \frac{ \mr{Var}(X_N(\Omega))}{{N}^{(\alpha /2)}} < + \infty .
            			\end{equation}
					     Moreover, $\frac{ \mr{Var}(X_N(\Omega))}{{N}^{(\alpha /2)}} $ is comparable to the Minkowski contents of $\pa \Omega$.
				\item [(iii)] If $\Omega$ is the Koch snowflake of a scale $\eta >1$ with $\eta$ satisfying the condition (\ref{Eqn: Non-Lattice Condition}), then the limit
			\begin{equation}
				\lim_{i \to \infty} \frac{ \mr{Var}(X_N(\Omega))}{{N}^{(\alpha(\eta) /2)}}
			\end{equation}
		  exists, where $\alpha(\eta)$ is the Minkowski dimension of $\pa \Omega$. Moreover, the limit is comparable to the Minkowski contents of $\pa \Omega$.
                \end{itemize}
            In all these three cases, we then obtain that the normalization of $X_N(\Omega)$ converges in law to the normal distribution.
            }
        \end{corollary}
	
	The proof of Corollary~\ref{Cor: Main Probability Models 2} will be shown in Section~\ref{Sec: Det Process}. In Section~\ref{Sec: Det Process}, we will also consider the limit of $\mr{Var}(X_N(\varphi))$ as $N \to +\infty$ when $\varphi \in W_0 ^{1,2}(B(0,\lambda))$ with $\lambda \in (0,1)$, which is a partial generalization of the results in~\cite{RV07-2}. It would also be intersting to build up similar central limit theorems for $\varphi \in W_0 ^{1,2}(\mb{R}^2)$, but we do not discuss further along this direction in this manuscript for compactness.

\begin{remark}\label{Rmk: J Change Signs}
	There are other point processes where our methods can also apply, but their variances may not have a nonnegative two-point function $J$ in the functional
		\begin{equation}
			\mc{J}_t(\Omega)  = \int_{\Omega} \int_{\Omega^c}   J \bigg(\frac{x-y}{t}\bigg) \ dx dy.
		\end{equation}
	For example, the zeros of the Gaussian entire function~~\cite{BS17} become a point process. The function is defined by $f(z) = \sum_{p=0}^{\infty} a_p e_p(z), z \in \mb{C}$,
	$a_p$'s are i.i.d.standard complex Gaussian random variables,  $\{e_p = z^p / \sqrt{p!}{\}}$ and it is an orthonormal basis of $L^2(\mb{C}, (e^{-{|z|}^2} / \pi) \ d\area(z))$, where $ d\area(z)$ is the Lebesgue measure when we regard $\mb{C}$ as $\mb{R}^2$.
	In this example, 
		\begin{equation}
			J(z) = J(|z|) = \frac{1}{2} \frac{d^2}{dr^2} \bigg|_{r = |z|}  r^2   \big[\coth{(r)} - 1 \big].
		\end{equation}
	See, for example,~\cite{BS17, FH99, SWY221}.
	This function $J$ will change its sign so the method proving Theorem~\ref{Thm: Main Upper Lower Bounds} does not apply to this case. However, we can separate the positive and negative part of this $J$ and apply Theorem~\ref{Thm: Main Finite Perimeter Limit} to those $\Omega \subset \mb{R}^2$ with finite perimeter.
	The limit $\lim_{t \to 0^+} \mc{J}_t(\Omega) / t^{3}$ will then depend on $\int_{\mb{R}^2} J(z) \cdot |z| \ d\area(z)$, which is still positive for the $J$ in this example. The method proving Theorem~\ref{Thm: Main Twistedflake limit} also applies to this example when $\Omega$ is a domain with a fractal boundary, while we do not know whether the limit is nonzero.

	Apart from this Gaussian entire function example, it would also be interesting to consider those functions $J$ such that the integrals $||J(z) \cdot |z|  ||_{L^1(\mb{R}^n)}$ vanish. 
	In~\cite{SWY222}, the authors also considered another similar nonlocal integral functional with a probability background, where they obtained results similar to Theorem~\ref{Thm: Main Finite Perimeter Limit} in the dimension $n=2$ case.
	It would also be intersting to see more applications of the nonlocal functional in~\cite{SWY222} in other fields like probability.
	But these are out of the scope of this manuscript and we leave those questions for further works.
\end{remark}

\section{Upper and Lower Limits in Theorem~\ref{Thm: Main Upper Lower Bounds}}\label{Sec: Bounds}

	Let $\Omega_1, \Omega_2 \subset \mb{R}^n$ be two  Lebesgue measurable sets and $\Omega_1 \cap \Omega_2 = \emptyset$. 
	We define
		\begin{equation}
			\Gamma_t \equiv \overline{B(\Omega_2 , t)} \cap \Omega_1 = \{ x \in \Omega_1 \ | \ \dist(x, \Omega_2) \leq t  {\}}. 
		\end{equation}
	Because $\Omega_1 \cap \Omega_2 = \emptyset$, $\Gamma_t$ also equals to $\overline{B( \pa \Omega_2 , t)} \cap \Omega_1$.
	We assume that $|\Omega_1|$ is bounded, where for any measurable set $A \subset \mb{R}^n$, we let $|A|$ be its Lebesgue measure.
	We assume assume that there is a $C_1>0$ and an $\alpha \in [0,n]$, such that $\big| \Gamma_t \big| \leq C_1 \cdot \min\{t^{n-\alpha} , 1{\}}$ for all $t>0$.
	Intuitively, this assumption is a little bit weaker than assuming that $\pa \Omega_1 \cap \pa \Omega_2$ is of Minkowski dimension $\alpha$.

	We assume that $J(z)$ for $z \in \mb{R}^n$ is a nonnegative function, such that 
		\begin{equation}
			\int_{\mb{R}^n} J(z) \cdot |z|^{n-\alpha} \ dz \leq C_J < \infty,
		\end{equation}
	for a positive constant $ C_J$. For simplicity, readers can keep natural examples like $J(z) \sim e^{-|z|}$ or $J(z) \sim e^{-{|z|}^2}$ in mind in the following.

	\begin{lemma}\label{Lem: General Upper Bound}
		\textit{
		For all $t>0$,
			\begin{equation}\label{Eqn: Upper Error}
				\int_{\Omega_1 \times \Omega_2} J \bigg( \frac{x-y}{t} \bigg) \ dxdy \leq C_1 \cdot  C_J \cdot t^{2n - \alpha}.
			\end{equation}
		}
	\end{lemma}
	\begin{proof}
		Notice that
			\begin{equation}
				\begin{split}
					& \quad \int_{\Omega_1 \times \Omega_2} J \bigg( \frac{x-y}{t} \bigg) \ dxdy 
				\\	&= \int_{\Omega_1} \int_{0} ^{\infty} \int_{\pa B(x,r) \cap \Omega_2}  J \bigg( \frac{x-y}{t} \bigg) \ d\mc{H}^{n-1}(y) dr dx.
				\end{split}
			\end{equation}
		As a function of $(x,r)$ on $\Omega_1 \times (0,+\infty)$,
			\begin{equation}
				\int_{\pa B(x,r) \cap \Omega_2}  J \bigg( \frac{x-y}{t} \bigg) \ d\mc{H}^{n-1}(y) \leq \mathds{1}_{\Gamma_r}(x) \cdot \int_{\pa B(x,r) }  J \bigg( \frac{x-y}{t} \bigg) \ d\mc{H}^{n-1}(y).
			\end{equation}
		Hence,
			\begin{equation}
				\begin{split}
					& \quad \int_{\Omega_1} \int_{0} ^{\infty} \int_{\pa B(x,r) \cap \Omega_2}  J \bigg( \frac{x-y}{t} \bigg) \ d\mc{H}^{n-1}(y) dr dx
				\\	& \leq \int_{\Omega_1} \int_{0} ^{\infty} \int_{\pa B(x,r)} \mathds{1}_{\Gamma_r}(x) J \bigg( \frac{x-y}{t} \bigg) \ d\mc{H}^{n-1}(y) dr dx
				\\	& = \int_{\Omega_1} \int_{\mb{R}^n} \mathds{1}_{\Gamma_{|z|}}(x) J \bigg( \frac{-z}{t} \bigg) \ dzdx
					= \int_{\mb{R}^n} \big| \Gamma_{|z|}\big| J \bigg( \frac{z}{t} \bigg) \ dz
				\\	& \leq C_1 \cdot \int_{\mb{R}^n} \min \big\{{|z|}^{n-\alpha} , 1 \big{\}} J \bigg( \frac{z}{t} \bigg) \ dz \leq C_1 \cdot \int_{\mb{R}^n} {|z|}^{n-\alpha}J \bigg( \frac{z}{t} \bigg) \ dz 
				\\	&\leq C_1 \cdot  C_J \cdot t^{2n - \alpha}.
				\end{split}
			\end{equation}

	\end{proof}

	Lemma~\ref{Lem: General Upper Bound}  provided a general approach for obtaining upper bounds. In the following, we will discuss potential lower bounds, subject to some additional natural assumptions for technical reasons.
	We first assume that $J(z) > c_J >0$ when $|z| \leq a_J$ for some positive constants $c_J, a_J$.
	We also assume that there is a part $E \subset \pa \Omega_1 \cap \pa \Omega_2$, and there is a $D_E  >0$, such that for all $z \in E$ and $t \in (0,1)$, $|B(z,t) \cap \Omega_2| \geq D_E \cdot t^n$.
	We may emphasize that this is the noncollapsing density at $z$ on the $\Omega_2$ side, and this is a natural assumption for NTA domains~~\cite{JK82, T17}.
	We also define $N_t(E) \equiv B(E , t) \cap \Omega_1  = \{ x \in \Omega_1 \ | \ \dist(x,E) < t{\}}$.

	\begin{lemma}\label{Lem: General Lower Bound}
		\textit{
		Suppose that there is a $C_2>0$ such that
			\begin{equation}
				|N_t(E)| \geq C_2 \cdot t^{n-\alpha},
			\end{equation}
		for all $t\in (0,1)$.
		Then there is a constant $c = c(n)>0$, such that for all $t \in (0,1)$,
			\begin{equation}
				\int_{\Omega_1} \int_{\Omega_2} J \bigg( \frac{x-y}{t} \bigg) \ dxdy \geq c(n) \cdot C_2 c_J D_E \cdot {({a_J t })}^{2n-\alpha}.
			\end{equation}
		}
	\end{lemma}
	\begin{proof}
		The proof is direct. For any $x \in N_{(a_J /2) \cdot t}(E)$, there is a $z \in E \cap B(x, (a_J/2) \cdot t) $. And we see that $B(z,(a_J/2) \cdot t) \cap \Omega_2 \subset B(x, a_J \cdot t)$, and $|B(z,(a_J/2) \cdot t) \cap \Omega_2| \geq D_E \cdot {(( a_J/2)\cdot t )}^n$ by our assumptions.
		For any $y \in B(z,(a_J/2) \cdot t) \cap \Omega_2$, we also know that $J \big( (x-y)/t \big) > c_J$.
		Hence,
			\begin{equation}
				\begin{split}
					& \quad \int_{\Omega_1} \int_{\Omega_2} J \bigg( \frac{x-y}{t} \bigg) \ dxdy
					 \geq \int_{N_{(a_J /2) \cdot t}(E) } \int_{B(z, (a_J/2) \cdot t) \cap \Omega_2}  J \bigg( \frac{x-y}{t} \bigg) \ dydx
				\\	& \geq \int_{N_{(a_J /2) \cdot t}(E) } \int_{B(z, (a_J/2) \cdot t) \cap \Omega_2}  c_J \ dydx
					 \geq c_J D_E {\bigg(\frac{a_J t}{2} \bigg)}^n \cdot \big | N_{(a_J /2) \cdot t}(E)  \big |
				\\	& \geq C_2c_J D_E {\bigg(\frac{a_J t}{2} \bigg)}^{2n-\alpha}.
				\end{split}
			\end{equation}
	\end{proof}

	We assume that $\Omega \subset \mb{R}^n$ is a bounded measurable set with $|\Omega| >0$, and with topological boundary $\pa \Omega$ such that 
	\begin{equation}
		0 < M_2 \leq \underline{\mc{M}}^{\alpha}(\pa \Omega) \leq \overline{\mc{M}}^{\alpha}(\pa \Omega) \leq M_1 < \infty, 
	\end{equation}
	where $M_1, M_2$ are two positive constants and $\alpha \in [n-1,n]$. The following Corollary then directly implies Theorem~\ref{Thm: Main Upper Lower Bounds}.

\begin{corollary}\label{Cor: Main Upper and Lower Limits}
	\textit{
	Under the same assumptions on $J(z)$ in Lemma~\ref{Lem: General Lower Bound},
	we also assume that there is a part $E \subset \pa \Omega$, on which for all $z \in E$ and $t \in (0,1)$, $|B(z,t) \cap \Omega| \geq D_E \cdot t^n$ and $|B(z,t) \cap \Omega^c| \geq D_E \cdot t^n$ for a positive constant $D_E$.
	Furthermore, we assume that $ \underline{\mc{M}}^{\alpha}(E) >0$. Then,
		\begin{equation}
			0 < \liminf_{t \to 0^+} \frac{1}{t^{2n-\alpha}} \int_{\Omega} \int_{\Omega^c} J \bigg( \frac{x-y}{t} \bigg) \ dxdy \leq \limsup_{t \to 0^+} \frac{1}{t^{2n-\alpha}} \int_{\Omega} \int_{\Omega^c} J \bigg( \frac{x-y}{t} \bigg) \ dxdy  < \infty,
		\end{equation}
	where the lower and upper limits are comparable to the Minkowski contents of $\pa \Omega$.
	}
\end{corollary}
\begin{proof}
	The upper limit follows from Lemma~\ref{Lem: General Upper Bound} directly. For the lower limit, notice that our assumption $ \underline{\mc{M}}^{\alpha}(E) >0$ is slightly weaker than the one we made in Lemma~\ref{Lem: General Lower Bound}. 
	The proof is exactly the same as the proof for Lemma~\ref{Lem: General Lower Bound}, and one needs to combine both parts $B(E,t) \cap \Omega$ and $B(E,t) \cap \Omega^c$ to get a similar lower bound in Lemma~\ref{Lem: General Lower Bound}.
	We can actually get that
		\begin{equation}
			\frac{1}{t^{2n-\alpha}} \int_{\Omega} \int_{\Omega^c} J \bigg( \frac{x-y}{t} \bigg) \ dxdy \geq C \cdot \frac{|B(\pa \Omega,t)|}{t^{n- \alpha}},
		\end{equation}
	where the positive constant $C$ depends on some parameters on $J$ and $\pa \Omega$ that we mentioned in Lemma~\ref{Lem: General Lower Bound}.
\end{proof}
\begin{remark}
	Recall that for Lemma~\ref{Lem: General Lower Bound} and Corollary~\ref{Cor: Main Upper and Lower Limits}, we assume that $J(z) > c_J >0$ when $|z| \leq a_J$. The condition that $J(z)> c_J$ on $|z| \leq a_J$ can also be replaced by $J(z)$ bounded from below on an annulus, but then one needs to have volume lower bounds better than $|B(x,t) \cap \Omega| \geq D_E \cdot t^n$ and $|B(x,t) \cap \Omega^c| \geq D_E \cdot t^n$, i.e., one also needs to replace balls $B(x,t)$ by annuluses.
\end{remark}

When $J(z)$ has the form ${|z|}^{-\beta}$, the author of~~\cite{L19} introduced a fractal dimension and discussed its relation with the Minkowski dimension. This singular form of $J(z)$ comes from the definition of fractional Sobolev space, as discussed in Section~\ref{Sec: Introduction} and~~\cite{BBM01}.
We will revisit the results in~~\cite{BBM01} in Section\ref{Sec: All Limits}.


\section{Limits in Theorem~\ref{Thm: Main Finite Perimeter Limit} and Theorem~\ref{Thm: Main Twistedflake limit}}\label{Sec: All Limits}
\subsection{Sets of finite perimeter and general BV functions}\label{Sec: Limit BV}

	\indent

	Let us first recall the results from~~\cite{BBM01, D02}. We will make some revisions to ensure that the settings in~~\cite{BBM01,D02} are compatible with this manuscript. 
	Assume that $J(z)$ for $z \in \mb{R}^n$ is nonnegative and is radially symmetric, i.e., $J(z) = J(|z|)$.
	And we assume that $\int_{\mb{R}^n} J(z) \cdot |z| \ dz = \int_0 ^{\infty} J(r)\cdot r^n \ dr = 1$, i.e., $\alpha = n-1$ in Section~\ref{Sec: Bounds}.
	Recall that the functional $\mc{J}_t(\Omega)$ we considered in Section~\ref{Sec: Bounds} has the following more symmetric form. For a measurable set $\Omega \subset \mb{R}^n$,
		\begin{equation}
			\int_{\Omega} \int_{\Omega^c} J \bigg( \frac{x-y}{t} \bigg) \ dxdy + \int_{\Omega^c} \int_{\Omega} J \bigg( \frac{x-y}{t} \bigg) \ dxdy= \int_{\mb{R}^n} \int_{\mb{R}^n} \big| \mathds{1}_{\Omega} (x) - \mathds{1}_{\Omega} (y) \big| J \bigg( \frac{x-y}{t} \bigg) \ dxdy .
		\end{equation}
	We can consider a general function $f(x)$ instead of only $\mathds{1}_{\Omega} (x)$, and we know the following theorem from~~\cite{BBM01}.
	\begin{theorem}[\cite{BBM01}, Theorem 3']\label{Thm: BV Criterion}
		\textit{
		Let $A \subset \mb{R}^n$ be a bounded smooth domain, $f \in L^1(A)$. Then $f \in \mr{BV}(A)$ if and only if 
			\begin{equation}
				\liminf_{t \to 0^{+}} \frac{1}{t^{n+1}} \int_A \int_A |f(x) - f(y)| \cdot J \bigg( \frac{x-y}{t} \bigg) \ dxdy < \infty. 
			\end{equation}
		}
	\end{theorem}
	Here, for a bounded smooth domain $A \subset \mb{R}^n$,  we say that a function $f \in \mr{BV}(A)$ if and only if ${[f]}_{\mr{BV}(A)}< \infty$, where the $\mr{BV}$-seminorm ${[f]}_{\mr{BV}(A)}$ is defined by 
		\begin{equation}
			{[f]}_{\mr{BV}(A)} \equiv \sup \bigg\{ \int_A f \di(\varphi) \ \bigg| \ \varphi \in C_0 ^{\infty}(A, \mb{R}^n), ||\varphi||_{\infty} \leq 1 \bigg{\}}.
		\end{equation}
	This $\mr{BV}$-seminorm equals to $\int_A |\nabla f |$ if $f \in W^{1,1}(A)$, which means that $\mr{BV}(A)$ can be seen as an extension of $W^{1,1}(A)$. In particular, $\mr{BV}(A)$ contains indicator functions for domains $\Omega \subset A$ with well-behaved topological boundaries $\pa \Omega$. 
	For instance, if $\Omega \subset \subset A$ has a Lipschitz boundary $\pa \Omega$, then the divergence theorem implies that $\mathds{1}_{\Omega} \in \mr{BV}(A)$ with ${[\mathds{1}_{\Omega}]}_{\mr{BV}(A)}  = \mc{H}^{n-1}(\pa \Omega)$, even though $\mathds{1}_{\Omega} \notin W^{1,1}(A)$. 
	
	\begin{definition}\label{Def: Finite Perimeter}
		Let $\Omega$ be a Borel set in $\mb{R}^n$. We call $\Omega$ a Caccioppoli set (or a set of locally finite perimeter) if and only if for every bounded open set $A \subset \mb{R}^n$, ${[\mathds{1}_{\Omega}]}_{\mr{BV}(A)} < \infty$.
	\end{definition}

	Theorem~\ref{Thm: BV Criterion} provides a direct criterion for determining whether a function belongs to the $\mr{BV}$ class.
	Additionally, De Giorgi~~\cite{De53,De54} actually used the limit $\lim_{t \to 0^+} \mc{J}_t(\Omega) / t^{n+1}$ with $J(z) = e^{-{|z|}^2}$ as his definition of perimeter for sets of locally finite perimeter up to a constant. 
	As we have seen, when $\Omega$ is a domain with a Lipschitz boundary, we have that ${[\mathds{1}_{\Omega}]}_{\mr{BV}(A)}  = \mc{H}^{n-1}(\pa \Omega)$.
	In general, for every $f \in \mr{BV}(A)$, the limit in Theorem~\ref{Thm: BV Criterion} actually exists, as shown in the following theorem.

	\begin{theorem}[\cite{D02}, Theorem 1]\label{Thm: BV Limit}
		\textit{
		Let $A \subset \mb{R}^n$ be a bounded domain with a Lipschitz boundary $\pa A$, and let $f \in \mr{BV}(A)$. Then,
			\begin{equation}
				\lim_{t \to 0^{+}} \frac{1}{t^{n+1}} \int_A \int_A |f(x) - f(y)| \cdot J \bigg( \frac{x-y}{t} \bigg) \ dxdy = K(n) \cdot {[ f]}_{\mr{BV}(A)},
			\end{equation}
		where
			\begin{equation}\label{Eqn: K(n) for Sets of Finite Perimeter}
				K(n) = \frac{\Gamma(\frac{n}{2})}{\sqrt{\pi} \Gamma(\frac{n+1}{2})}.
			\end{equation}
		}
	\end{theorem}

	We then see that Theorem~\ref{Thm: BV Limit} holds true for $f = \mathds{1}_{\Omega}$ when $\Omega$ is a Caccioppoli set.
	
	\begin{corollary}\label{Cor: Finite Perimeter Limit}
		\textit{
		For the same $A, K(n)$ in Theorem\ref{Thm: BV Limit}, if $\Omega$ is a Caccioppoli set, then
			\begin{equation}
				\lim_{t \to 0^{+}} \frac{1}{t^{n+1}} \int_{\Omega \cap A} \int_{\Omega^c \cap A}   J \bigg( \frac{x-y}{t} \bigg) \ dxdy = 2 K(n) \cdot \mc{H}^{n-1}(\pa^* \Omega \cap A) .
			\end{equation}
		}
	\end{corollary}

	Here, $\pa^* \Omega$ is called the reduced boundary of $\Omega$, which is a subset of $\pa \Omega$ on which the density of $\Omega$ in balls centered at those points is $1/2$. This is a consequence of De Giorgi's structure theorem. See Corollary 15.8 and Theorem 15.9 in~~\cite{M12}, for example.
	For $\Omega$ with a Lipschitz boundary, $\pa^* \Omega$ includes those differentiable points on $\pa \Omega$.
	We can also generalize Corollary~\ref{Cor: Finite Perimeter Limit} slightly.
	
	\begin{corollary}\label{Cor: Finite Perimeter Limit of Two Sets}
		\textit{
		For the same $A, K(n)$ in Theorem\ref{Thm: BV Limit}, if $\Omega_1$ and $\Omega_2$ are two disjoint Caccioppoli sets, then
			\begin{equation}
				\begin{split}
					& \quad \lim_{t \to 0^{+}} \frac{1}{t^{n+1}} \int_{\Omega_1 \cap A} \int_{\Omega_2 \cap A}  J \bigg( \frac{x-y}{t} \bigg) \ dxdy 
				\\	&=  2K(n) \cdot \mc{H}^{n-1}( \pa^* \Omega_1 \cap \pa^* \Omega_2 \cap A) .
				\end{split}
			\end{equation}
		}
	\end{corollary}
	\begin{proof}
		Notice that because $J$ is radially symmetric, and $\Omega_1 \cap \Omega_2 = \emptyset$,
			\begin{equation}
				\begin{split}
					&\quad 2\int_{\Omega_1 \cap A} \int_{\Omega_2 \cap A}  J \bigg( \frac{x-y}{t} \bigg) \ dxdy  
				\\	&= \int_{\Omega_1 \cap A} \int_{\Omega_1 ^c \cap A}  J \bigg( \frac{x-y}{t} \bigg) \ dxdy + \int_{\Omega_2 \cap A} \int_{\Omega_2 ^c \cap A}  J \bigg( \frac{x-y}{t} \bigg) \ dxdy 
				\\	& \quad - \int_{(\Omega_1 \cup \Omega_2) \cap A} \int_{{(\Omega_1 \cup \Omega_2)}^c \cap A}  J \bigg( \frac{x-y}{t} \bigg) \ dxdy . 
				\end{split}
			\end{equation}
		And by Corollary\ref{Cor: Finite Perimeter Limit}, we know that the limit is
			\begin{equation}
				2K(n) \cdot (\mc{H}^{n-1}( \pa^* \Omega_1 \cap A) +  \mc{H}^{n-1}( \pa^* \Omega_2 \cap A) -  \mc{H}^{n-1}( \pa^* (\Omega_1\cup \Omega_2) \cap A)).
			\end{equation}
		Because $\Omega_1 \cap \Omega_2 = \emptyset$, Theorem 16.3 in~~\cite{M12} implies that $\pa^* (\Omega_1\cup \Omega_2)$ is, up to a set of $\mc{H}^{n-1}$ measure zero, $(\pa^* \Omega_1 \backslash \pa^* \Omega_2) \cup (\pa^* \Omega_2 \backslash \pa^* \Omega_1)$.
		Hence, the above equation equals to
			\begin{equation}
				4K(n) \cdot \mc{H}^{n-1}( \pa^* \Omega_1 \cap \pa^* \Omega_2 \cap A)  .
			\end{equation}
	\end{proof}


\subsection{Sobolev Spaces}\label{Sec: Limit Sobolev}
	There is another result in~~\cite{BBM01} that deals with the case for the functional
		\begin{equation}
			\int_A \int_A {|f(x) - f(y)|}^p \cdot J \bigg( \frac{x-y}{t} \bigg) \ dxdy,
		\end{equation}
	where $1< p <\infty$. 
	To be more precise, we assume that $J(z)$ for $z \in \mb{R}^n$ is nonnegative and is radially symmetric, i.e., $J(z) = J(|z|)$.
	We further assume that $\int_{\mb{R}^n} J(z) \cdot {|z|}^p \ dz = \int_0 ^{\infty} J(r)\cdot r^{n+ p - 1} \ dr = 1$. One can compare these assumptions with those we made in Section~\ref{Sec: Bounds} and also in Section~\ref{Sec: Limit BV}.

	\begin{definition}
		Let $A \subset \mb{R}^n$ be a bounded smooth domain, $f \in L^p(A)$ and $1 \leq p < \infty$. We say that $f \in W^{1,p}(A)$ if and only if ${[ f]}_{W^{1,p}(A)} < \infty$. Here, the seminorm ${[ f]}_{W^{1,p}(A)} $ is defined by
			\begin{equation}
				{[ f]}_{W^{1,p}(A)}  = {\bigg(\int_A {|\nabla f |}^p \ dx\bigg)}^{\frac{1}{p}},
			\end{equation}
		where $\nabla f$ is the weak derivative of $f$.
	\end{definition}

	Then, we have the following theorem from~~\cite{BBM01} when $1 < p < \infty$.
	\begin{theorem}[\cite{BBM01}, Theorem 2]\label{Thm: W^{1,p} Limits}
		\textit{
		Let $A \subset \mb{R}^n$ be a bounded smooth domain, $f \in L^p(A)$ and $1< p < \infty$. Then,
			\begin{equation}
				\lim_{t \to 0^{+}} \frac{1}{t^{n+p}} \int_A \int_A {|f(x) - f(y)|}^p \cdot J \bigg( \frac{x-y}{t} \bigg) \ dxdy  = K(n,p) \cdot { \big({[ f]}_{W^{1,p}(A)} \big)}^p,
			\end{equation}
		with the convention that ${[ f]}_{W^{1,p}(A)} = \infty$ if $f \notin W^{1,p}(A)$. Here, $K(n,p)$ is a constant depending on $n,p$.
		}
	\end{theorem}


\subsection{Koch snowflake and sets of self-similar fractal boundary}\label{Sec: Limit Fractal}

	\indent

	In stead of discussing sets with general self-similar fractal boundaries in detail, we will focus on a specific example, Koch snowflake (Figure~\ref{Fig: Snowflake}), in this section. 
	But the idea is actually applicable to many other sets with self-similar boundaries constructed in a similar way to the Koch snowflake.
	And we will provide an alternative example, constructed using a similar method as the Koch snowflake but with a different scale (Figure~\ref{Fig: Twistedflake}). 
	Figure~\ref{Fig: Constructflake} shows the first several steps of the construction of these shrunk Koch snowflakes  (or Koch curves).
	In order to maintain consistency with Section~\ref{Sec: Bounds}, we will take $n=2$ in this subsection, but continue to use the notation $n$ when referring to the dimension.

	\begin{figure}[ht!]
		\centering
		\includegraphics[width=80mm]{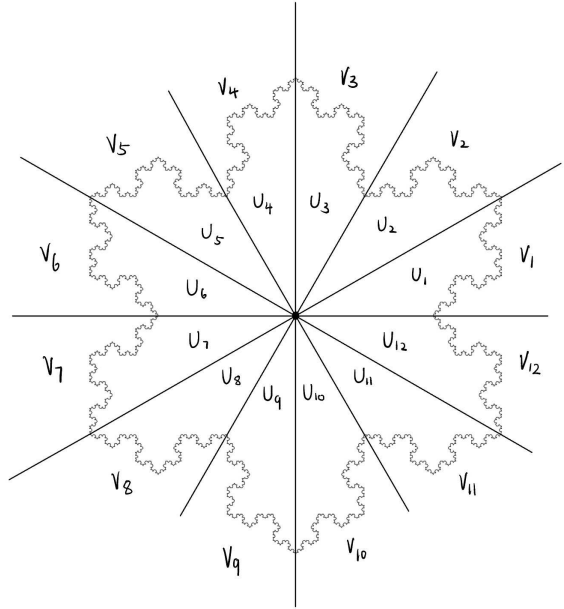}
		\caption{Standard Koch Snowflake\label{Fig: Snowflake}}
	\end{figure}

	\begin{figure}[ht!]
		\centering
		\includegraphics[width=110mm]{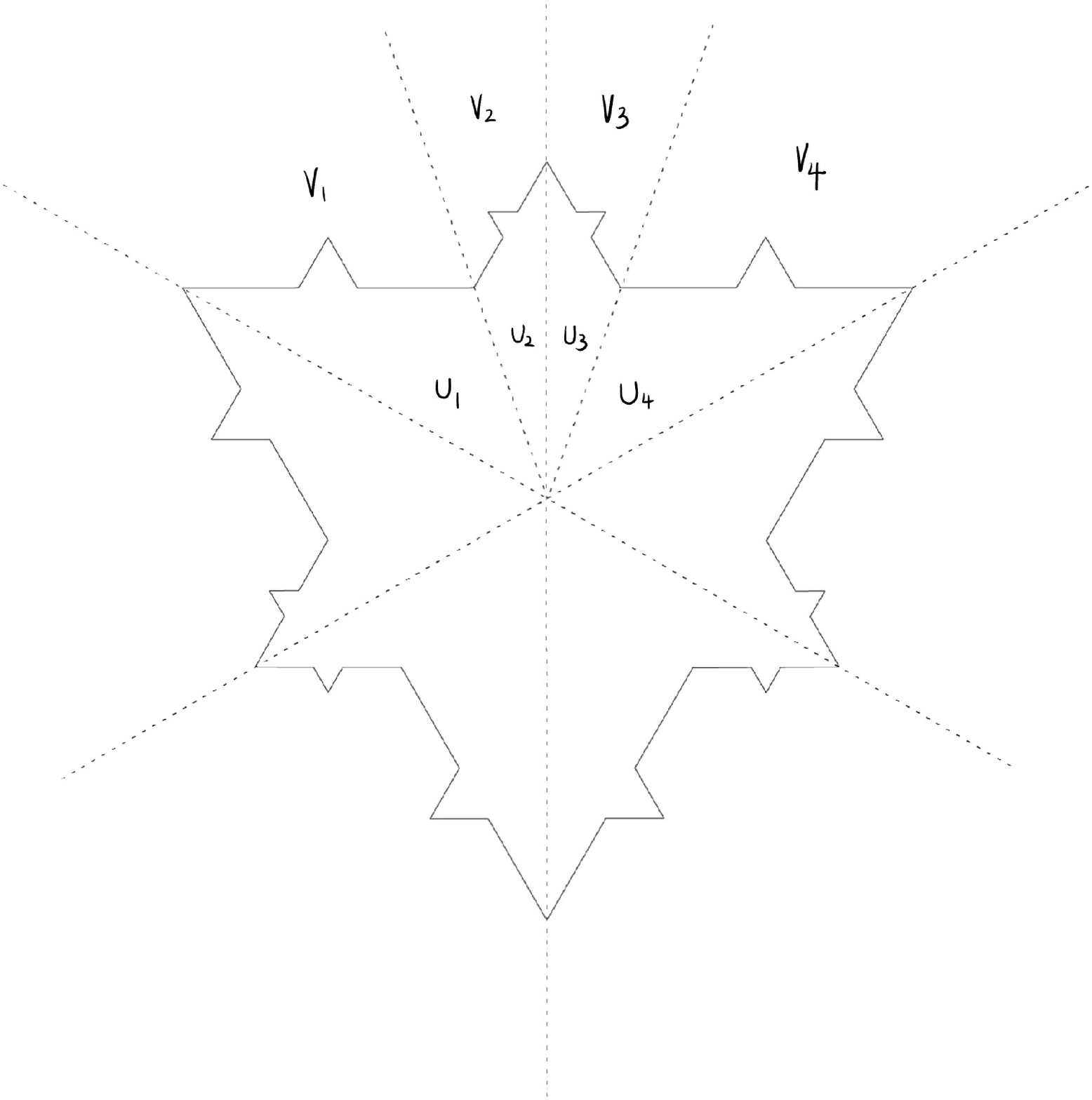}
		\caption{Snowflake with $\eta = 5$\label{Fig: Twistedflake}}
	\end{figure}

	\begin{figure}[ht!]
		\centering
		\includegraphics[width=80mm]{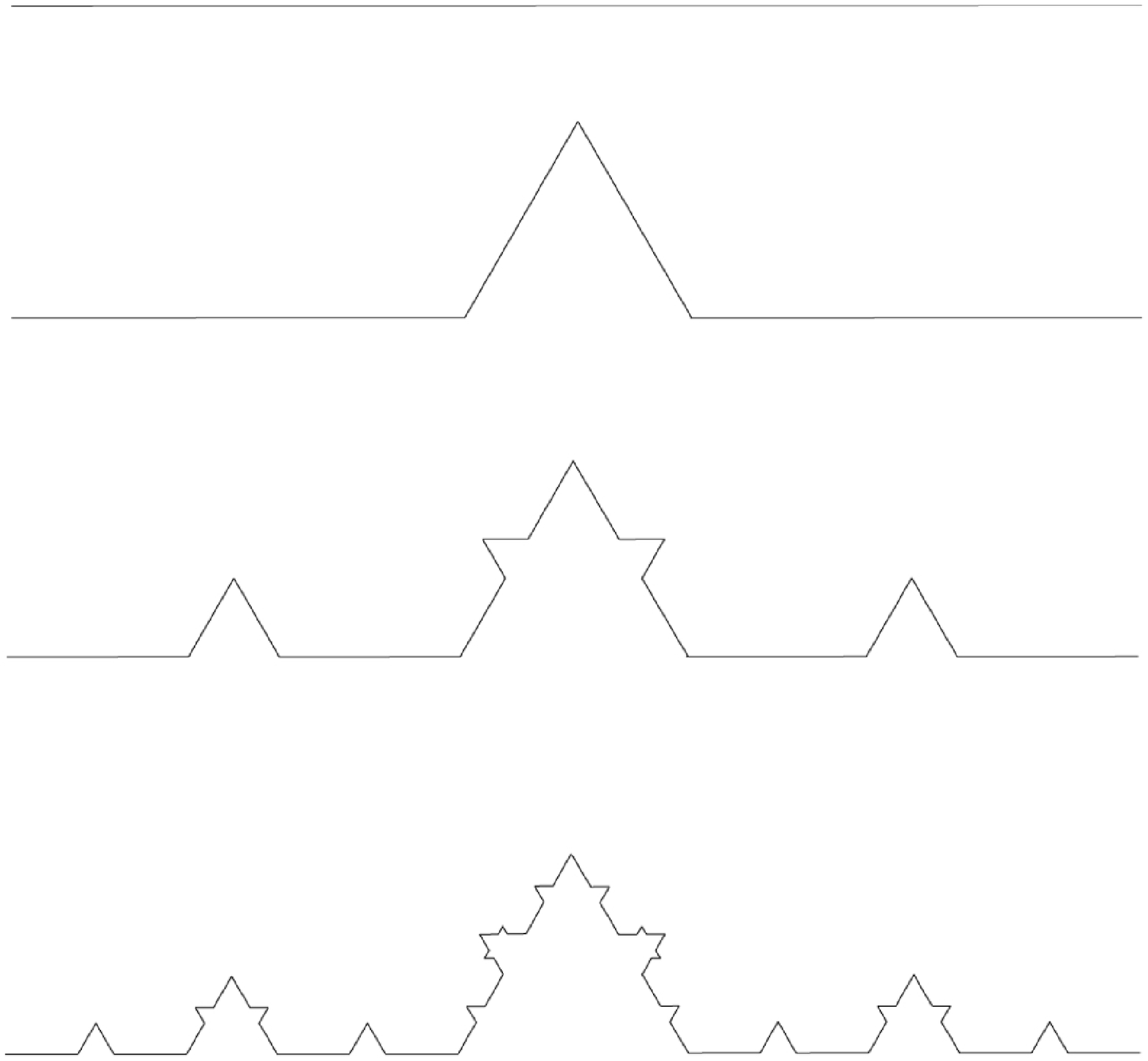}
		\caption{Koch Curve with $\eta = 5$. We do this iterative process on each side of a equilateral triangle as shown in Figure~\ref{Fig: Twistedflake}.\label{Fig: Constructflake}}
	\end{figure}

	The standard Koch snowflake, which we say is of scale $\eta = 3$, is constructed iteratively. The process starts with an equilateral triangle of side length $1$. Then, on the middle of each side of this triangle, we add a smaller equilateral triangle of side length $1/\eta=1/3$, which results in a regular hexagram.
	We repeat this process iteratively, adding an equilateral triangle of side length ${(1/3)}^2$ to the middle of each side of the hexagram. Taking the limit of this process yields an open set with a boundary of Minkowski dimension (and also Hausdorff dimension) $\alpha(3) \equiv \log_3(4) $. For further details, see, for example,~~\cite{F03}.

        A general Koch snowflake of scale $\eta>1$ is constructed in a similar manner, as illustrated in Figure~\ref{Fig: Constructflake}.
	The iterative process is carried out on each side of an equilateral triangle, where a small equilateral triangle of side length $1/\eta$ is added to the middle of the original line segment (one side of the original equilateral triangle). Then we obtain $4$ line segments. 
	On each of these line segments, a smaller equilateral triangle with side length $1/\eta \times (\text{side length of this line segment})$ is then added to its middle, generating a total of $4^2$ line segments.
	The process is repeated by adding smaller equilateral triangles to the new line segments.
	The limit open set is the Koch snowflake of scale $\eta$, whose boundary has Minkowski dimension $\alpha(\eta)$ satisfying 
            \begin{equation}\label{Eqn: Dimension of Twistedflake}
                2\cdot {\bigg(\frac{\eta-1}{2\eta}\bigg)}^{\alpha(\eta)} + 2 \cdot {\bigg(\frac{1}{\eta}\bigg)}^{\alpha(\eta)} = 1.
            \end{equation}
        $\alpha(\eta)$ satisfies this equation because of the self-similarity. Let us see Figure~\ref{Fig: Twistedflake}.
	The boundary part between $U_1$ and $V_1$ is $(\eta-1)/2$ times larger than the boundary part between $U_2$ and $V_2$, while the boundary part between $(U_1 \cup U_2 \cup U_3 \cup U_4)$ and $(V_1 \cup V_2 \cup V_3 \cup V_4)$ is $\eta$ times larger than the boundary part between $U_2$ and $V_2$.
	Additionally, the the boundary part between $(U_1 \cup U_2 \cup U_3 \cup U_4)$ and $(V_1 \cup V_2 \cup V_3 \cup V_4)$ is the sum of twice the boundary part between $U_1$ and $V_1$ and twice the boundary part between $U_2$ and $V_2$.

	Notice that the Koch snowflake satisfies the property we need to prove Corollary~\ref{Cor: Main Upper and Lower Limits}, i.e., it has positive finite lower and upper bounds on Minkowski contents, together with density lower bounds at the boundary of the Koch snowflake. 
	This is because the Koch snowflake is a quasiball and therefore an NTA domain.  For example, see Remark 2.1 in the survey paper by~~\cite{T17}. 
	Additionally, in~~\cite{M10}, some $3$-dimensional analogs of snowflakes were also shown to be quasiballs.
	In particular, in the $2$-dimensional case, Ahlfors's $3$-point condition states that a Jordan curve $\gamma$ in $\mb{R}^2$ is a quasicircle if and only if for any two points $a , b$ on $\gamma$, the smaller arc between $a,b$ has a diameter comparable to the distance $|a-b|$.
	
	In~~\cite{G00}, similar limiting behaviors of Minkowski contents were established for general self-similar fractals.

	\begin{theorem}\label{Thm: Snowflake Limit}
		\textit{
		Assume that $J(z)$ for $z \in \mb{R}^n$ is a nonnegative function in $L^1(\mb{R}^n)$ such that 
			\begin{equation}
				\int_{\mb{R}^n} J(z) \cdot |z|^{n} \ dz \leq C_J < \infty,
			\end{equation}
		for a positive constant $ C_J$. Let $\Omega = \Omega(3)$ be the standard Koch snowflake. Then, for any $t >0$, the limit
			\begin{equation}
				\lim_{i \to \infty}{(3^i/t)}^{2n- \alpha(3)} \int_{\Omega \times \Omega^c} J \bigg( \frac{x-y}{(t/3^i)} \bigg) \ dxdy 
			\end{equation}
		exists and is positive.
		}
	\end{theorem}
	\begin{proof}
		The positivity of the limit follows from Lemma~\ref{Lem: General Lower Bound}. 
		For the existence, we let $\Omega$ be the bounded domain in Figure~\ref{Fig: Snowflake} of Koch snowflake, and we let $\Omega^c$ be the unbounded part. We evenly partition $\Omega$ into $U_1, \ldots, U_{12}$, and evenly partition $\Omega^c$ into $V_1, \ldots, V_{12}$ as shown in Figure~\ref{Fig: Snowflake}.
		This division is chosen to facilitate the computation of the functional since the boundary part of $\pa U_1 \cap \pa V_1$ is a rescaling by $1/3$ of the boundary part of $\pa (U_{12} \cup U_1 \cup U_2 \cup U_3) \cap \pa (V_{12} \cup V_1 \cup V_2 \cup V_3)$. 
		Therefore, the functional can be expressed as
			\begin{equation}\label{Eqn: Snowflake Decomposition}
				\begin{split}
					\int_{\Omega \times \Omega^c} J \bigg( \frac{x-y}{t} \bigg) \ dxdy = \sum_{i=1} ^{12} \int_{U_i \times V_i} J \bigg( \frac{x-y}{t} \bigg) \ dxdy + \sum_{i \neq j} \int_{U_i \times V_j} J \bigg( \frac{x-y}{t} \bigg) \ dxdy.
				\end{split}
			\end{equation}
		Let us first see the second term in (\ref{Eqn: Snowflake Decomposition}). If $i \neq j$, we see that either $\dist(U_i, V_j) >0$, like $U_1$ and $V_3$, or $\pa U_i \cap \pa V_j$ is a single point, like $U_1$ and $V_2$.
		
		For the pair $(U_1, V_3)$, one can similarly define $\Gamma_t$ like we did for Lemma~\ref{Lem: General Upper Bound} in Section~\ref{Sec: Bounds}.
		$\Gamma_t = \emptyset$ when $t$ is small, so there is a $D_1(\Omega)>0$ such that $|\Gamma_t| < D_1 \cdot \min \{t^{2n} , 1{\}}$. So, Lemma~\ref{Lem: General Upper Bound} shows that this pair will give a term less than order $t^{2n}$.
		
		For the pair $(U_1, V_2)$ (or the pair $(U_1,V_{12})$), notice that $U_1$ is included in a cone $\widetilde{U}_1$ and $V_2$ is included in a cone $\widetilde{V}_2$. Also, $\widetilde{U}_1$ and $\widetilde{V}_2$ touch at their vertices.
		Again, one can similarly define $\Gamma_t$, and there is a $D_2(\Omega)>0$ such that $|\Gamma_t| < D_2 \cdot \min \{t^{2n} , 1{\}}$. So, Lemma~\ref{Lem: General Upper Bound} shows that this pair will give a term less than order $t^{2n}$.

		Hence, there is a $D(\Omega)>0$ such that the second term in (\ref{Eqn: Snowflake Decomposition}) is less than $C_J \cdot D\cdot t^{2n}$. 
		
		Next, let us consider the first term in (\ref{Eqn: Snowflake Decomposition}). Since all $\{U_i{\}}$'s (and all $\{V_i{\}}$'s) are the same up to rigid transformations, the first term in (\ref{Eqn: Snowflake Decomposition}) is 
			\begin{equation}
				12 \cdot \int_{U_1 \times V_1} J \bigg( \frac{x-y}{t} \bigg) \ dxdy,
			\end{equation}
		which we denote as $12 F(t)$, where $F(t)$ is this functional involving domains $U_1, V_1$. 
		Because $\pa U_1 \cap \pa V_1$ is a rescaling by $1/3$ of the boundary part of $\pa (U_{12} \cup U_1 \cup U_2 \cup U_3) \cap \pa (V_{12} \cup V_1 \cup V_2 \cup V_3)$, $\pa (3 U_1) \cap \pa (3 V_1) = \pa (U_{12} \cup U_1 \cup U_2 \cup U_3) \cap \pa (V_{12} \cup V_1 \cup V_2 \cup V_3)$.
		Let us consider the pair of domains $(3U_1, 3V_1)$, and the functional
			\begin{equation}
				\int_{3U_1 \times 3V_1} J \bigg( \frac{x-y}{t} \bigg) \ dxdy.
			\end{equation}
		Compare the pair of domains $(3U_1,3V_1)$ with the pair of domains $(U_{12} \cup U_1 \cup U_2 \cup U_3 , V_{12} \cup V_1 \cup V_2 \cup V_3)$.
		By using a similar method to estimate the second term in (\ref{Eqn: Snowflake Decomposition}), up to a term less than $C_J \cdot D \cdot t^{2n}$, the above functional of the pair $(3U_1,3V_1)$ equals to 
			\begin{equation}
				\int_{(U_{12} \cup U_1 \cup U_2 \cup U_3) \times (V_{12} \cup V_1 \cup V_2 \cup V_3)} J \bigg( \frac{x-y}{t} \bigg) \ dxdy,
			\end{equation}
		which is also equal to
			\begin{equation}
				4 \cdot \int_{U_1 \times V_1} J \bigg( \frac{x-y}{t} \bigg) \ dxdy
			\end{equation}
		up to another term less than $C_J \cdot D \cdot t^{2n}$.
		On the other hand, by change of variables,
			\begin{equation}
				\int_{3U_1 \times 3V_1} J \bigg( \frac{x-y}{t} \bigg) \ dxdy = 3^4 \cdot \int_{U_1 \times V_1} J \bigg( \frac{x-y}{(t/3)} \bigg) \ dxdy.
			\end{equation}
		So, $ 4F(t) = 3^4 F(t/3) + R(t)$ for an $R(t)$ such that $|R(t)| \leq C_J \cdot D \cdot t^{2n}$. Recall that in our case, $n = 2$ and $\alpha(3) = \log_3 (4)$, hence $3^4/4 = 3^{2n -\alpha(3)}$. Because $t>0$ is arbitrary, so for any $i \in \mb{N}$,
			\begin{equation}
				{(3^{2n- \alpha(3)})}^i F(t/3^i) = {(3^{2n- \alpha(3)})}^{i+1} F(t/3^{i+1}) + {(3^{2n- \alpha(3)})}^i \cdot R(t/3^i).
			\end{equation}
		We see that
			\begin{equation}
				\sum_{i=0} ^\infty  {(3^{2n- \alpha(3)})}^i \cdot |R(t/3^i)| \leq \sum_{i=0} ^\infty  {(3^{2n- \alpha(3)})}^i \cdot C_J \cdot D(\Omega) \cdot t^{2n} \cdot {(3^{2n})}^{-i} < \infty.
			\end{equation}
		So, the limit
			\begin{equation}
				\lim_{i \to \infty}{(3^{2n- \alpha(3)})}^i F(t/3^i)
			\end{equation}
		exists for any $t >0$.
	\end{proof}
		It is unknown whether the limit
			\begin{equation}
				\lim_{t \to 0^+} {\frac{1}{t^{2n- \alpha(3)}}} \int_{\Omega \times \Omega^c} J \bigg( \frac{x-y}{t} \bigg) \ dxdy 
			\end{equation}
		exists for the standard Koch snowflake. On the other hand, it is known that the lower and upper Minkowski contents of the standard Koch snowflake are not the same.

		However, for a snowflake of scale $\eta > 1$, if the logarithm ratio of the two factors $1/\eta$ and $(\eta-1)/(2\eta)$ is not a rational number, i.e., 
			\begin{equation}\label{Eqn: Non-Lattice Condition}
				\frac{\log[(2\eta) / (\eta-1)]}{\log(\eta)} = 1 + \frac{\log(2) - \log(\eta-1)}{\log(\eta)} \notin \mb{Q},
			\end{equation}
		then the limiting behaviors are actually better in some sense. Indeed, we have the following theorem. Also notice that any transcendental number satisfies the condition (\ref{Eqn: Non-Lattice Condition}).

	\begin{theorem}\label{Thm: Twistedflake limit}
		\textit{
		Assume that $J(z)$ for $z \in \mb{R}^n$ is a nonnegative function in $L^1(\mb{R}^n)$ such that 
			\begin{equation}
				\int_{\mb{R}^n} J(z) \cdot |z|^{n} \ dz \leq C_J < \infty,
			\end{equation}
		for a positive constant $ C_J$.
		Let $\Omega = \Omega(\eta)$ be the Koch snowflake of scale $\eta>1$. If $\eta$ satisfies (\ref{Eqn: Non-Lattice Condition}), then the limit 
			\begin{equation}
				\lim_{t \to 0^+} {\frac{1}{t^{2n- \alpha(\eta)}}} \int_{\Omega \times {\Omega}^c} J \bigg( \frac{x-y}{t} \bigg) \ dxdy 
			\end{equation}
		exists.
		}
	\end{theorem}
	\begin{proof}
		Let us use Figure~\ref{Fig: Twistedflake} to illustrate the proofs, which follow the same strategy as the proof for Theorem~\ref{Thm: Snowflake Limit}. 
		We take domains $U_1,U_2,U_3,U_4$ and $V_1,V_2,V_3,V_4$ as shown in Figure\ref{Fig: Twistedflake}.
		And we only need to prove the existence for the limit
			\begin{equation}
				\lim_{t \to 0^+} {\frac{1}{t^{2n- \alpha(\eta)}}} \int_{(U_1 \cup U_2 \cup U_3 \cup U_4) \times (V_1 \cup V_2 \cup V_3 \cup V_4)} J \bigg( \frac{x-y}{t} \bigg) \ dxdy .
			\end{equation}
		If we let 
			\begin{equation}
				F(t) \equiv \int_{(U_1 \cup U_2 \cup U_3 \cup U_4) \times (V_1 \cup V_2 \cup V_3 \cup V_4)} J \bigg( \frac{x-y}{t} \bigg) \ dxdy,
			\end{equation}
		then, use the self-similarity as we obtained the dimension $\alpha(\eta)$ in (\ref{Eqn: Dimension of Twistedflake}) again, we will get that for all $t>0$,
			\begin{equation}
				2 \bigg[F\bigg(\frac{t }{(1-(1/\eta))/2}\bigg) \cdot {\bigg(\frac{1-(1/\eta)}{2}\bigg)}^{2n} + F\bigg(\frac{t}{(1/\eta)}\bigg) \cdot {\bigg(\frac{1}{\eta}\bigg)}^{2n}\bigg] = F(t) + R(t),
			\end{equation}
		where there is a constant $D = D(\Omega(\eta))>0$ such that $|R(t)| \leq C_J \cdot D \cdot t^{2n}$.

		Define $G(s) \equiv e^{(2n -\alpha(\eta))s} \cdot F(e^{-s})$, and let $r_1 = (1-(1/\eta))/2$, $r_2 = 1/\eta$, we see that 
			\begin{equation}
				2 \bigg[G\big(s + \log(r_1) \big) \cdot {r_1}^{\alpha(\eta)} + G\big(s + \log(r_2) \big) \cdot {r_2}^{\alpha(\eta)}   \bigg] = G(s) + e^{(2n-\alpha(\eta))s}R(e^{-s}).
			\end{equation}
		Let $\mu$ be the probability distribution function which assigns weight $2 r_i ^{\alpha(\eta)}$ at the point $-\log(r_i)$ for $i=1,2$. Then, the function $G(s)$ satisfies the renewal equation 
			\begin{equation}
				G(s) = - e^{(2n-\alpha(\eta))s}R(e^{-s}) + \int_{0} ^s G(s- s') \ d\mu(s').
			\end{equation}
		Because we assumed that $\log(r_1) / \log(r_2) \notin \mb{Q}$, we can apply the renewal theorem, as was done in~~\cite{G00} for the limits of Minkowski contents. For general statements of the renewal theorem, see Chapter XI of~~\cite{F71}.
		Therefore, we can show that
			\begin{equation}
				\lim_{s \to + \infty} G(s) = \frac{1}{2r_1^{\alpha(\eta)} \log(r_1) +2r_2^{\alpha(\eta)} \log(r_2) } \cdot \int_{0} ^{+\infty}  e^{(2n-\alpha(\eta))s}R(e^{-s}) \ ds .
			\end{equation}
		This completes the proof of Theorem~\ref{Thm: Twistedflake limit}.
	\end{proof}

	\begin{remark}
		As shown in Lemma~\ref{Lem: General Upper Bound} and Lemma~\ref{Lem: General Lower Bound}, both limits in Theorem~\ref{Thm: Snowflake Limit} and Theorem~\ref{Thm: Twistedflake limit} are comparable, i.e., up to a constant depending on the dimension, the kernel $J$, and the noncollapsing parameter $D_{\pa \Omega}$, to the Minkowski contents of the corresponding $\pa \Omega$'s. And in these specific Koch snowflakes, $D_{\Omega}$ can also be chosen explicitly depending on the scale $\eta$.
	\end{remark}


\section{Applications in Fluctuations of Determinantal Processes}\label{Sec: Det Process}

	For simplicity, let us consider random point processes on $(\mb{R}^n, \mu)$, where $\mu$ is a probability measure on $\mb{R}^n$. For more rigorous definitions, see Section 4.2 of~~\cite{AGZ10}. 
	We let $\mf{M}$ be the space of $\sigma$-finite Radon measures on $\mb{R}^n$, and let $\mf{M}_+$ be the subset of $\mf{M}$ consisting of positive measures.
	\begin{definition}\label{Def: Point Process}
		A random point process on $\mb{R}^n$ is a random, integer-valued $X \in \mf{M}_+$.
	\end{definition}
	By random, we mean that for any Borel set $A \subset \mb{R}^n$, $X(A)$ is an integer-valued random variable.
	
	\begin{definition}\label{Def: Cor Funct}
		Assume that there are locally nonnegative integrable symmetric functions $\rho_k$ on $ \mb{R}^{n \times k}$ for all $k \geq 1$, such that for any measurable subset $A$ of $\mb{R}^n$,
			\begin{equation}
				k! \cdot \mc{E} \bigg[ \binom{X(A)}{k}   \bigg]  = \int_{A^{\otimes k} } \rho_k(x_1,\ldots, x_k)  \ d\mu(x_1) \ldots d\mu(x_k).
			\end{equation}
		Here, $\mc{E}(\cdot)$ denotes expectations. And we say that $\rho_k$ is symmetric if
			\begin{equation}
				\rho_k(x_1, \ldots, x_k) = \rho_k(x_{\sigma(1)}, \ldots, x_{\sigma(k)}),
			\end{equation}
		for all $\sigma \in S_k$, the $k$-th symmetry group. And we call the function $\rho_k$ $k$-th joint intensity or $k$-th correlation function of the point process $X$ with respect to $(\mb{R}^n , |\cdot|)$. Here, $|\cdot|$ denotes the Lebesgue measure on $\mb{R}^n$.
	\end{definition}

	\begin{definition}\label{Def: Det Process}
		$X$ is called a determinantal point process if there is a kernel function $K$ on $\mb{R}^n \times \mb{R}^n$ such that for all $k \geq 1$,
			\begin{equation}
				\rho_k(x_1, \ldots, x_k) = \det {\big( K(x_i,x_j) \big)}_{1 \leq i,j \leq k} .
			\end{equation}
	\end{definition}
	
	\begin{example}\label{Example: Ginibre}
		Let $N \in \mb{N}$. For the \textit{complex Ginibre ensemble} of dimension $N$, we mean $N \times N$ random matrices over $\mb{C}$, with entries being i.i.d.~complex Gaussian random variables with mean $0$ and variance $1/N$. 
		Let $(\lambda_1 ^N, \ldots, \lambda_N ^N)$ be the eigenvalues of the Ginibre ensemble of dimension $N$. Then, the law for those eigenvalues is 
			\begin{equation}
				p(z_1, \ldots, z_N) = \frac{1}{Z_N} \cdot \prod_{1 \leq i < j \leq N } {|z_i - z_j|}^2 \cdot  e^{- N \sum_{i=1} ^N  {|z_i|}^2 } \ d\mc{H}^2(z_1) \ldots d\mc{H}^2(z_N).
			\end{equation}
		Here, each $z_i$ is in $\mb{C}$, and $\mc{H}^2$ is the $2$-dimensional Hausdorff measure on $\mb{C}$, which is exactly the Lebesgue measure on $\mb{C} = \mb{R}^2$. $Z_N$ is a constant such that the integral of $p$ is $1$.
		Then, if we let $X_N = \sum_{i=1} ^N \delta_{z_i} $, where $\delta_z$ is the Dirac measure at $z$, then $X_N$ is a determinantal process, such that the kernel function is
			\begin{equation}
				K_N(z,w) = \sum_{m = 0} ^{N-1} \frac{N^m}{m !} {(z \overline{w})}^m, \quad z,w \in \mb{C},
			\end{equation}
		with respect to the measure $d\mu_N (z) = \frac{N}{\pi} e^{-N {|z|}^2} d\mc{H}^2(z)$. Notice that $K_N(w,z) = \overline{K_N(z,w)}$. For more properties of the Ginibre ensemble, one may see~~\cite{RV07-2}.
	\end{example}

	\subsection{Asymptotic behaviors of eigenvalues of the Ginibre ensemble}

	\indent

	We will let $N \to \infty$ and study the behavior of $X_N$, which is the sum of Dirac measures on the eigenvalues of the Ginibre ensemble as defined in Example~\ref{Example: Ginibre}.
	First, we have the following lemma for the kernel function $K_N$ in Example~\ref{Example: Ginibre}.

	\begin{lemma}\label{Lem: Exponential Remainder}
		\textit{
		For any $\lambda \in (0,1)$, there is a positive constant $C(\lambda)$ and a constant $\delta(\lambda) \in (0,1)$, such that for any $z,w \in \mb{C}$ with $|z|, |w| < \lambda$, 
			\begin{equation}
				\big| e^{N(z\overline{w})} - K_N(z,w) \big|  \leq C\cdot \frac{e^{(N\cdot |z\overline{w}|)}}{\sqrt{N}} \cdot \delta^N .
			\end{equation}
		}
	\end{lemma}
	\begin{proof}
		For any fixed $N$ and $s \in (0,\lambda)$,
			\begin{equation}
			    \begin{split}
			        & \quad e^{Ns} - \sum_{m = 0} ^{N-1} \frac{N^m}{m !} {s}^m = \sum_{m \geq N} \frac{N^m}{m !} {s}^m 
        		\\	&= \frac{ {( Ns )}^{N} }{N !} \cdot \bigg( 1+\sum_{m \geq 1} \frac{{( Ns )}^{m}}{(N+1) \cdots (N+m)} \bigg) 
				\leq C(\lambda)\cdot  \frac{ {( Ns )}^{N} }{N !}  ,
			    \end{split}
			\end{equation}	
            	where the last ineqaulity is because $N \leq N+i$ for $i\geq 0$ and $\sum_{m\geq 1} s^m \leq C(\lambda)$.
		Notice that Stirling's formula showed that there is another dimensional constant $C>0$, such that $k ! \geq C \sqrt{k} {(k/e)}^k$ for any $k \in \mb{Z}_+$. 
		Hence,
			\begin{equation}
				\frac{ {( Ns )}^{N} }{N !} \leq \frac{{(se)}^N}{C \sqrt{N}} .
			\end{equation}
		Notice that the function for $s$, $s e^{1-s}$, is increasing in $(0,1)$. So, $se \leq \lambda e^{1+s - \lambda}$. What's more, $\lambda e^{1-\lambda} < 1 e^{1-1} = 1$ since $\lambda \in (0,1)$. We may then denote the constant $\lambda e^{1-\lambda}$ as $\delta(\lambda) <1$. Then, $se \leq \delta(\lambda) e^s$. So, 
			\begin{equation}
				\bigg| e^{Ns} - \sum_{m = 0} ^N \frac{N^m}{m !} {s}^m  \bigg| \leq \frac{e^{Ns}}{C\sqrt{N}} \cdot \delta^N.
			\end{equation}

	\end{proof}

        We first identify $\mb{C}$ with $\mb{R}^2$. Then, let us consider a function $\varphi \in L^1(\mb{R}^2) $ with the closure of its support $\overline{\mr{supp}(\varphi)} \subset \big\{x \in \mb{R}^2 \ \big| \ |x| < 1 \big{\}}$. Then, 
        	\begin{equation}
			X_N(\varphi) \equiv \sum_{i=1} ^N \varphi(z_i)
		\end{equation}
        is a random variable and we will study its asymptotic behavior as $N \to \infty$. 
	A typical choice for $\varphi$ is to take $\varphi = \mathds{1}_{\Omega}$ for some domain $\Omega$ with closure contained in the unit ball of $\mb{R}^2$. In the following, we assume that there exists a $\lambda \in (0,1)$ such that $\overline{\mr{supp}(\varphi)} \subset \big\{x \in \mb{R}^2 \ \big| \ |x| < \lambda \big{\}}$, and let $\delta(\lambda) \in (0,1)$ be the constant we obtained in Lemma~\ref{Lem: Exponential Remainder}.

        \begin{lemma}
		\textit{
		There is a dimensional constant $C>0$, such that
			\begin{equation}
					\bigg| \mc{E}\big(X_N(\varphi) \big) - \frac{N}{\pi} \int \varphi(x) \ dx\bigg| \leq  C \cdot (\sqrt{N}  \delta^N )\cdot ||\varphi||_{L^1(\mb{R}^2)}.
			\end{equation}
		}
        \end{lemma}
        \begin{proof}
            Because $X_N$ is a determinantal process, the definition and Lemma~\ref{Lem: Exponential Remainder} can give us that 
                \begin{equation}
                    \begin{split}
                        & \quad \bigg|  \mc{E}\big(X_N(\varphi) \big) - \frac{N}{\pi} \int \varphi(x) \ dx\bigg| = \bigg| \frac{N}{\pi}\int_{\mb{R}^2} \varphi(x) \cdot ( K_N(x,x) - e^{N {|x|}^2}) e^{-N {|x|}^2} \ dx \bigg|
                    \\  &\leq C \cdot N \int_{\mb{R}^2} |\varphi| \cdot  \frac{\delta^N}{\sqrt{N}} \ dx
                        = C \cdot (\sqrt{N}  \delta^N )\cdot ||\varphi||_{L^1(\mb{R}^2)}.
                    \end{split}
                \end{equation}
        \end{proof}

	Next, we need the formula for the variance of $X_N(\varphi)$. We cite the equation (28) in~~\cite{RV07}. For $\varphi \in L^2(\mb{R}^2)$,
		\begin{equation}\label{Eqn: Variance General L^2 Functions}
			\mr{Var}(X_N(\varphi)) = \frac{1}{2}\int_{\mb{R}^2 \times \mb{R}^2} { \big|\varphi(x) - \varphi(y) \big| }^2 \cdot {|K_N(x,y)|}^2   \ d\mu_N(x)  d\mu_N(y).
		\end{equation}
	In particular, when $\varphi = \mathds{1}_{\Omega}$,
		\begin{equation}\label{Eqn: Variance Single Set}
			\mr{Var}(X_N(\mathds{1}_{\Omega})) = \int_{\Omega \times \Omega^c} {|K_N(x,y)|}^2   \ d\mu_N(x)  d\mu_N(y).
		\end{equation}

	Recall that we assume that there exists $\lambda \in (0,1)$ such that $\overline{\mr{supp}(\varphi)} \subset \big\{x \in \mb{R}^2 \ \big| \ |x| < \lambda \big{\}}$, and we can get the constant $\delta(\lambda) \in (0,1)$ in Lemma~\ref{Lem: Exponential Remainder}.  Let us choose a $\lambda_+ \in (\lambda , 1)$.
	We see that since $\varphi = 0$ outside of $B(0,\lambda)$, we have that
		\begin{equation}
			\begin{split}
				\mr{Var}(X_N(\varphi)) &= \frac{1}{2}\int_{\mb{R}^2 \times \mb{R}^2} { \big|\varphi(x) - \varphi(y) \big| }^2 \cdot {|K_N(x,y)|}^2   \ d\mu_N(x)  d\mu_N(y)
			\\	&= \frac{1}{2}\int_{B(0,\lambda_+)\times B(0,\lambda_+)} { \big|\varphi(x) - \varphi(y) \big| }^2 \cdot {|K_N(x,y)|}^2   \ d\mu_N(x)  d\mu_N(y)
			\\	& \quad + \int_{B(0,\lambda_+)\times {(B(0,\lambda_+))}^c } { \big|\varphi(x) \big| }^2 \cdot {|K_N(x,y)|}^2   \ d\mu_N(x)  d\mu_N(y) .
			\end{split}
		\end{equation}
	Let us compute the second term in the above formula. We use the information of Ginibre ensemble and we see that since $\varphi(x) = 0$ outside of $B(0,\lambda)$,
		\begin{equation}
			\begin{split}
				& \quad \int_{B(0,\lambda_+)\times {(B(0,\lambda_+))}^c } { \big|\varphi(x) \big| }^2 \cdot {|K_N(x,y)|}^2   \ d\mu_N(x)  d\mu_N(y)
			\\	&= \int_{B(0,\lambda)\times {(B(0,\lambda_+))}^c } { \big|\varphi(x) \big| }^2 \cdot {|K_N(x,y)|}^2   \ d\mu_N(x)  d\mu_N(y)
			\\	&\leq {\bigg( \frac{N}{\pi} \bigg)}^2 \int_{B(0,\lambda)\times {(B(0,\lambda_+))}^c } { \big|\varphi(x) \big| }^2 \cdot {e^{2N |x| |y| }} \cdot (e^{-N{|x|}^2} e^{-N{|y|}^2})  \ dx  dy
			\\	&\leq {\bigg( \frac{N}{\pi} \bigg)}^2 \int_{B(0,\lambda)\times {(B(0,\lambda_+))}^c } { \big|\varphi(x) \big| }^2 \cdot e^{-N{(|y|- \lambda)}^2}  \ dx  dy
			\\	&= {\bigg( \frac{N}{\pi} \bigg)}^2 \cdot {||\varphi||}^2 _{L^2(\mb{R}^2)} \cdot (2\pi) \cdot \int_{\lambda_+ - \lambda} ^{\infty} e^{-Nr^2}r \ dr
			\\	&= {\bigg( \frac{N}{\pi} \bigg)} \cdot {||\varphi||}^2 _{L^2(\mb{R}^2)} \cdot e^{-N {(\lambda_+ - \lambda)}^2}.
			\end{split}
		\end{equation}
	Hence, the second term is an exponentially small term. For the first term, we need to use Lemma~\ref{Lem: Exponential Remainder} to replace ${|K_N(x,y)|}^2$ with ${|e^{N (x \overline{y})}|}^2$. 
	By Lemma~\ref{Lem: Exponential Remainder}, we see that when $x,y \in B(0,\lambda_+)$, for suitably large $N$ depending on $\lambda_+$,
		\begin{equation}
			\bigg| {|K_N(x,y)|}^2 -  {|e^{N (x \overline{y})}|}^2\bigg| \leq C(\lambda_+) \cdot \frac{e^{(N\cdot |x||y|)}}{\sqrt{N}} \cdot {(\delta(\lambda_+))}^N \cdot 3 e^{(N\cdot |x||y|)} .
		\end{equation}
	So, recall that $d\mu_N(x) = \frac{N}{\pi} e^{-N {|x|}^2} \ dx$, we know that
		\begin{equation}
			\begin{split}
				& \quad \int_{B(0,\lambda_+)\times B(0,\lambda_+)} { \big|\varphi(x) - \varphi(y) \big| }^2 \cdot \bigg| {|K_N(x,y)|}^2 -  {|e^{N (x \overline{y})}|}^2\bigg|   \ d\mu_N(x)  d\mu_N(y)
			\\	&\leq \frac{6}{\pi^2} C(\lambda_+) \cdot  {(\delta(\lambda_+))}^N \cdot N^{\frac{3}{2}} \cdot {||\varphi||}^2 _{L^2(\mb{R}^2)} \cdot |B(0,\lambda_+)|,
			\end{split}
		\end{equation}
	which is also an exponentially small term. Hence, up to a term $R_N$ which is exponentially small as $N \to +\infty$,
		\begin{equation}\label{Eqn: Asymptotic Ginibre Variance}
			\begin{split}
				\mr{Var}(X_N(\varphi))  &=  \frac{1}{2}\int_{B(0,\lambda_+)\times B(0,\lambda_+)} { \big|\varphi(x) - \varphi(y) \big| }^2 \cdot {|e^{N (x \overline{y})}|}^2   \ d\mu_N(x)  d\mu_N(y) + R_N
			\\	&= \frac{N^2}{2 \pi^2}\int_{B(0,\lambda_+)\times B(0,\lambda_+)} { \big|\varphi(x) - \varphi(y) \big| }^2 \cdot e^{-N {|x-y|}^2}   \ dx  dy + R_N.
			\end{split}
		\end{equation}
	Therefore, we can directly obtain the following corollary by applying Theorem~\ref{Thm: W^{1,p} Limits}.
	\begin{corollary}\label{Cor: Ginibre Sobolev}
		\textit{
		There is a dimensional constant $T_1>0$, such that for any $\lambda \in (0,1)$, and any $\varphi \in W_0 ^{1,2}(B(0,\lambda))$, we have that 
			\begin{equation}
				\lim_{N \to \infty} \mr{Var}(X_N(\varphi)) = T_1 \int_{\mb{R}^2} {|\nabla \varphi|}^2 \ dx.
			\end{equation}
		}
	\end{corollary}
	Notice that similar results with $\varphi \in W^{1,2}$ sense for the Gaussian kernel $e^{-N {|x-y|}^2}$ were already built up in~~\cite{RV07}.
	Additionally, the authors of~~\cite{RV07} discussed other kernels $K(x,y)$ that are not of the form $K(x-y)$ but can be bounded by another kernel $J(x-y)$.
	For the Ginibre ensemble, in~~\cite{RV07-2}, it was shown that for $\varphi \in C_0 ^1(\mb{R}^2)$,
		\begin{equation}
			\lim_{N \to \infty} \mr{Var}(X_N(\varphi)) = \frac{1}{4\pi} {||\nabla \varphi ||}^2 _{L^2(\mb{R}^2)} + \frac{1}{2} {|| \varphi ||}^2 _{H^{1/2}( \pa B(0,1))}.
		\end{equation}
	So, $T_1$ in Corollary~\ref{Cor: Ginibre Sobolev} is $1/ 4\pi$. And the authors of~\cite{RV07-2} also built up a central limit theorem for the normalizations of $X_N(\varphi)$, which did not follow from the methods of~\cite{CL95,S00,AGZ10} because $\mr{Var}(X_N(\varphi))$ is bounded now.
	It would also be interesting to extend those results in~\cite{RV07-2} to all $\varphi \in W_0 ^{1,2} (\mb{R}^2)$, which requires a deeper analysis of the behavior of the kernel $K_N(x,y)$ for the Ginibre ensemble near the boundary $\pa B(0,1)$.

	Now, consider $\varphi = \mathds{1}_{\Omega}$ with $\overline{\Omega} \subset B(0,1) \subset \mb{R}^2$. Notice that in this case, ${|\varphi(x) - \varphi(y)|}^2 = |\varphi(x) - \varphi(y)|$. Therefore, if $\Omega$ is a set of finite perimeter (bounded Caccioppoli set), Corollary~\ref{Cor: Finite Perimeter Limit} together with (\ref{Eqn: Asymptotic Ginibre Variance}) will imply the following corollaries.
	\begin{corollary}\label{Cor: Ginibre Finite Perimeter 1}
		\textit{
		There is a dimensional constant $T_2>0$, such that
			\begin{equation}
				\lim_{N \to \infty} \frac{ \mr{Var}(X_N(\mathds{1}_{\Omega}))}{\sqrt{N}} = T_2 \cdot \mc{H}^{1}(\pa^* \Omega).
			\end{equation}
		}
	\end{corollary}
	Or by Corollary~\ref{Cor: Finite Perimeter Limit of Two Sets}, we also have the following.
	\begin{corollary}\label{Cor: Ginibre Finite Perimeter 2}
		\textit{
		If $\Omega_1$ and $\Omega_2$ are two disjoint sets of finite perimeter with closures contained in $B(0,1) \subset \mb{R}^2$, then
			\begin{equation}
				\lim_{N \to \infty} \frac{ \mr{Cov}(X_N(\mathds{1}_{\Omega_1}) , X_N(\mathds{1}_{\Omega_2}) )}{\sqrt{N}} = T_2 \cdot \mc{H}^{1}(\pa^* \Omega_1 \cap \pa^* \Omega_2).
			\end{equation}
		}
	\end{corollary}
	Let $\Omega = \Omega(\eta)$ be the Koch snowflake of scale $\eta>1$ in Section~\ref{Sec: Limit Fractal} with $\eta$ satisfying the condition (\ref{Eqn: Non-Lattice Condition}). 
	Then, by Theorem~\ref{Thm: Twistedflake limit} and (\ref{Eqn: Asymptotic Ginibre Variance}), we know the following. 
	\begin{corollary}\label{Cor: Ginibre Fractal}
		\textit{
		There is a positive constant $T( \Omega)$ depending on $\Omega$, which is also compatible with the Minkowski content of $\pa \Omega$, such that, 
			\begin{equation}
				\lim_{i \to \infty} \frac{ \mr{Var}(X_N(\mathds{1}_{\Omega}))}{{N}^{(\alpha(\eta) /2)}} = T( \Omega),
			\end{equation}
		where $\alpha(\eta) \in (1,2)$ satisfies (\ref{Eqn: Dimension of Twistedflake}) and is the Minkowski dimension of $\pa \Omega$.
		}
	\end{corollary}
    Finally, if $\Omega \in \mb{R}^n$ satisfies the assumptions in Theorem~\ref{Thm: Main Upper Lower Bounds} (equivalently, Corollary~\ref{Cor: Main Upper and Lower Limits}), then we can also get part (ii) of Corollary~\ref{Cor: Main Probability Models 2} by using (\ref{Eqn: Asymptotic Ginibre Variance}) and Corollary~\ref{Cor: Main Upper and Lower Limits}.

    \bigskip

    {\bf Acknowledgements} 
    The author would like to thank Professor Paul Bourgade for introducing him to the probabilistic background of this problem and for the insightful discussions.
    The author's advisor, Professor Fang-Hua Lin, brought~~\cite{CRS10} and Professor Guido De Philippis brought~~\cite{De53,De54} to the author's attention, and the author also wants to thank them for their continuous help and support at Courant.
    The author also thanks Oren Yakir for his very helpful comments on the early draft of the manuscript.






\bibliographystyle{siam} 
\bibliography{LimitContentEJP.bib}

\end{document}